\documentclass[letterpaper,11pt]{amsart}
\usepackage[margin=1in]{geometry}
\usepackage{tikz}
\usetikzlibrary{decorations.markings}
\usetikzlibrary{shapes}
\usetikzlibrary{backgrounds}
\usepackage{subfig}
\usepackage{amssymb} 
\usepackage[all]{xy}

\usepackage{hyperref}
\def\thetitle{{Surface subgroups of graph products of groups}}
\hypersetup{
    pdftitle=   \thetitle,
   pdfauthor=  {Sang-hyun Kim}
}
\usepackage{enumerate}
\makeatletter
\let\@@enum@org\@@enum@
\def\@@enum@[#1]{\@@enum@org[\normalfont #1]}
\makeatother
\newcounter{saveenum}

\newtheorem{thm}{Theorem}
\newtheorem{lem}[thm]{Lemma}
\newtheorem{cor}[thm]{Corollary}

\newtheorem{que}[thm]{Question}
\newtheorem{prob}[thm]{Problem}

\def\be{\begin{enumerate}}
\def\ee{\end{enumerate}}

\theoremstyle{remark}
\newtheorem{exmp}[thm]{Example}
\newtheorem*{rem}{Remark}
\newtheorem*{notelit}{Note on the literature}

\theoremstyle{definition}
\newtheorem{defn}[thm]{Definition}

\newcommand\form[1]{\langle #1\rangle}
\newcommand\ssm{\smallsetminus}
\def\opp{^\mathrm{opp}}
\newcommand\fform[1]{\langle\!\langle #1\rangle\!\rangle}

\newcommand\co{\colon}

\newcommand\R{\mathbb{R}}

\newcommand\Z{\mathbb{Z}}
\newcommand\val[1]{\operatorname{val}(#1)}

\newcommand\lk{\operatorname{Lk}}
\newcommand\st{\operatorname{St}}

\newcommand\Gam{\Gamma}
\newcommand\gam{\gamma}

\newcommand\mS{\mathcal{S}}
\newcommand\mX{\mathcal{X}}
\newcommand\mG{\mathcal{G}}

\newcommand\mH{\mathcal{H}}
\newcommand\mV{\mathcal{V}}

\title\thetitle

\author{Sang-hyun Kim}
\address{Department of Mathematical Sciences, KAIST, Yuseong-gu, Daejeon 305-701, Republic of Korea}
\email{shkim@kaist.edu}
\thanks{The author is supported by the Basic Science Research Program (2011-0026138) and the Mid-Career Researcher Program (2011-0027600) through the National Research Foundation funded by the Ministry of Education, Science and Technology of Korea.}

\date{\today}
\begin{document}
\begin{abstract}
A graph product kernel means the kernel of the natural surjection from a graph product to the corresponding direct product. We prove that a graph product kernel of countable groups is special, and a graph product of finite or cyclic groups is virtually cocompact special in the sense of Haglund and Wise. The proof of this yields conditions for a graph over which the graph product of arbitrary nontrivial groups (or some cyclic groups, or some finite groups) contains a hyperbolic surface group. In particular, the graph product of arbitrary nontrivial groups over a cycle of length at least five, or over its opposite graph, contains a hyperbolic surface group. For the case when the defining graphs have at most seven vertices, we completely characterize right-angled Coxeter groups with hyperbolic surface subgroups.
 \end{abstract}
\maketitle

\section{Introduction}

By a \emph{graph}, we mean a simplicial $1$--complex. 
Throughout this paper, we will let $\Gam$ be a finite graph. 
The vertex set and the edge set of $\Gam$ are denoted as $V(\Gam)$ and $E(\Gam)$, respectively. 
Suppose $\mG=\{G_v\co v\in V(\Gam)\}$ is a collection of groups indexed by $V(\Gam)$.
We define $GP(\Gam,\mG)$ to be the free product of the groups in $\mG$ quotient by the normal closure of the set $\{ [g,h] \co  g\in G_u,h\in G_v\mbox{ for some  }\{u,v\}\in E(\Gam)\}$. 
We call $GP(\Gam,\mG)$ as the \emph{graph product of the groups in $\mG$ over $\Gam$},
and each $G_v$ as a \emph{vertex group} of $GP(\Gam,\mG)$.
The kernel of the natural surjection $GP(\Gam,\mG)\to \prod_{v\in V(\Gam)}G_v$ is called as the \emph{graph product kernel of $\mG$ over $\Gam$} and denoted as $KP_0(\Gam,\mG)$.

By a \emph{hyperbolic surface group}, we mean the fundamental group of a closed hyperbolic surface.
For abbreviation, we let $\mS$ be the class of groups that contain hyperbolic surface groups.
Our main question is the following.
\begin{que}\label{que:main}
For which graph $\Gam$ and which collection of groups $\mG$, is $GP(\Gam,\mG)$ in $\mS$?
\end{que}

Let us briefly explain some motivation for Question~\ref{que:main}. Gromov asked the following intriguing question~\cite[p.277]{gromov1987}.

\begin{que}\label{que:gromov}
Is every one-ended word-hyperbolic group in $\mS$?
\end{que}

Question~\ref{que:gromov} has been answered for only a few cases, all affirmatively.
These include graphs of free groups with cyclic edge groups with nontrivial second rational homology~\cite{calegari2008}, doubles of rank-two free groups symmetrically amalgamated along cyclic edge groups~\cite{GW2010,KW2009,KO2011}, and most remarkably, the fundamental groups of closed hyperbolic $3$--manifolds~\cite{KM2009}. We note that these groups are all \emph{virtually cocompact special} in the sense that each one is virtually the fundamental group of a compact \emph{special} cube complex~\cite{HW2008,HW2011}\footnote{The fact that closed hyperbolic $3$--manifold groups are virtually cocompact special is recently announced by Agol.}; see Definition~\ref{defn:special}. So very broadly, we may ask under which conditions a one-ended, virtually cocompact special group belongs to $\mS$. On the other hand, 

\begin{thm}\label{thm:v special}
\be
\item
Graph product kernels of countable groups are special.
\item
Graph products of finite or cyclic groups are virtually cocompact special.
\ee
\end{thm}

The proof of Theorem~\ref{thm:v special} will reveal inclusion relations between certain subgroups of graph products, and so, provide an important tool for this paper. In some sense, a graph product kernel will ``remember'' only the order of each vertex group, while ``forgetting'' the group structure of it.

For $2\le m\le\infty$, we let $GP_m(\Gam)$ denote the graph product of cyclic groups of order $m$ over $\Gam$.
We write $A(\Gam)=GP_\infty(\Gam)$ and $C(\Gam)=GP_2(\Gam)$. We will call $A(\Gam)$ and $C(\Gam)$ as a \emph{right-angled Artin group} and a \emph{right-angled Coxeter group} on $\Gam$, respectively~\cite{charney2007}.
Question~\ref{que:main} has a close relation to the question of whether $A(\Gam)\in\mS$ or $C(\Gam)\in\mS$ as described below.

\begin{thm}\label{thm:artin coxeter}
\begin{enumerate}
\item
We have $C(\Gam)\in\mS$ if and only if the graph product of arbitrary nontrivial groups over $\Gam$ is in $\mS$.
\item
We have $A(\Gam)\in\mS$ if and only if the graph product of some cyclic groups over $\Gam$ is in $\mS$.
\item
We have $[A(\Gam),A(\Gam)]\in\mS$ if and only if the graph product of some finite groups over $\Gam$ is in $\mS$, if and only if $GP_m(\Gam)\in \mS$ for some $2\le m<\infty$.
\end{enumerate}
\end{thm}

We denote by $C_m$ the cycle of length $m$.
The \emph{opposite graph} $\Gam\opp$ of $\Gam$ is defined by $V(\Gam\opp)=V(\Gam)$ and $E(\Gam\opp)=\{\{u,v\}\co u\mbox{ and }v\mbox{ are non-adjacent vertices of }\Gam\}$.
If there is a finite sequence of edge-contractions~\cite[p.20]{diestel2010} from $\Gam_1\opp$ to $\Gam_2\opp$,
we say $\Gam_1$ \emph{co-contracts} onto $\Gam_2$. In~\cite{kim2008}, it was shown that a co-contraction $\Gam_1\to\Gam_2$ induces an embedding $A(\Gam_2)\hookrightarrow A(\Gam_1)$. 

\begin{thm}\label{thm:cocont}
Suppose $\Gam_1$ and $\Gam_2$ are finite graphs such that $\Gam_1$ co-contracts onto $\Gam_2$.
If $2\le m\le \infty$, then $GP_m(\Gam_2)$ embeds into $GP_m(\Gam_1)$.
\end{thm}

It is well-known that $C(C_m)$ and $A(C_m)$ are in $\mS$ for $m\ge5$; see~\cite{SDS1989}. 
Also, it was shown that $A(C_m\opp)\in\mS$ for $m\ge5$ in~\cite{kim2008,css2008}; see~\cite{bell2012} for an alternative proof. 
Using Theorem~\ref{thm:cocont}, we generalize these results.

\begin{cor}[\cite{kim2007,HR2011}, \emph{cf}.~\cite{FT2012}]\label{cor:cycle}
For $m\ge5$, the graph product of arbitrary nontrivial groups over $C_m$ is in $\mS$.
\end{cor}

\begin{cor}\label{cor:anticycle}
For $m\ge5$, the graph product of arbitrary nontrivial groups over $C_m\opp$ is in $\mS$.
\end{cor}

Suppose $X\subseteq V(\Gam)$. The \emph{induced subgraph of $\Gam$ on $X$} is the maximal subgraph of $\Gam$ whose vertex set is $X$.
If $\Lambda$ is isomorphic to an induced subgraph of $\Gam$, we simply write $\Lambda\le\Gam$ and say that $\Gam$ has an \emph{induced} $\Lambda$.
We also use the notation $H\le G$ for two groups $G$ and $H$, if there exists an embedding from $H$ into $G$.
We say $\Gam$ is \emph{weakly chordal} if $\Gam$ does not contain an induced $C_m$ or $C_m\opp$ for $m\ge5$. 
For each finite graph $\Gam_1$, there exists a (algorithmically constructible) graph $\Gam_2\ge\Gam_1$ such that $[C(\Gam_2):A(\Gam_1)]<\infty$~\cite{DJ2000}. In particular, $A(\Gam_1)\in\mS$ if and only if $C(\Gam_2)\in\mS$.
Hence, the classification of $\Gam$ satisfying $C(\Gam)\in\mS$ is presumably ``harder'' than that of $\Gam$ satisfying $A(\Gam)\in\mS$.
Complete classification of the graphs $\Gam$ with $|V(\Gam)|\le8$ and $A(\Gam)\in\mS$ is given in~\cite{css2008}.
We will classify all the graphs $\Gam$ with $|V(\Gam)|\le7$ and $C(\Gam)\in\mS$.

\begin{thm}\label{thm:7v}
Suppose $\Gam$ has at most seven vertices.
Then $C(\Gam)\in\mS$ if and only if $\Gam$ is not weakly chordal.
\end{thm}

In particular, the proof of Theorem~\ref{thm:7v} will exhibit graphs $\Gam$ such that $A(\Gam)\in\mS$ and $C(\Gam)\not\in\mS$.
When $\Gam$ has more than seven vertices, $C(\Gam)\in\mS$ does not necessarily imply that $\Gam\ge C_m$ or $\Gam\ge C_m\opp$ for some $m\ge5$; see Remark~\ref{rem:w chordal}.  Lastly, we will make an observation that the class of finitely generated groups that ``conform'' to an affirmative answer to Question~\ref{que:gromov} is closed under graph products. 

Here is the organization of this paper. In Section~\ref{sec:prelim}, we summarize basic facts on cube complexes and label-reading maps. 
We describe two special cube complexes whose fundamental groups are specific subgroups of graph products and use these complexes to prove Theorems~\ref{thm:v special},~\ref{thm:artin coxeter} and Corollary~\ref{cor:cycle} in Section~\ref{sec:cplx}. Section~\ref{sec:double} introduces a general, combinatorial group theoretic lemma, which yields nontrivial embeddings between  graph products. Theorem~\ref{thm:cocont} and Corollary~\ref{cor:anticycle} will follow. In Section~\ref{sec:7v}, we investigate seven-vertex graphs and prove Theorem~\ref{thm:7v}.  We discuss a role of graph products in relation to Question~\ref{que:gromov} in Section~\ref{sec:gp closure}.

\begin{notelit}
\be
\item
Haglund has shown that a graph product of \emph{finite} groups is virtually cocompact special~\cite{haglund2008} by considering a certain cube complex for its graph product kernel; see also~\cite{CRSV2010}. The cube complex discussed in Section~\ref{sec:cplx} is not a generalized version of his complex.
\item
While it is unknown whether Coxeter groups are virtually \emph{cocompact} special, they are virtually special~\cite[Problem 9.2, Theorem 1.2]{HW2011}.
This already implies that graph products of finite or cyclic groups are virtually special, since these graph products embed into Coxeter groups. Note that for Coxeter groups, Question~\ref{que:gromov} has an affirmative answer as well~\cite{GLR2004}.
\item
Holt and Rees constructed a complex $Z$ for a graph product kernel of \emph{cyclic} groups~\cite[Theorem 3.1]{HR2011}. Their complex $Z$ is different from ours in that $Z$ is not cubical and not necessarily aspherical.
Theorem~\ref{thm:artin coxeter} (1) can also be proved using the construction of Holt and Rees, combined with Droms' description of a complex for $[C(\Gam),C(\Gam)]$; see~\cite{droms2003}.
\item
Corollary~\ref{cor:cycle} was first proved in the author's~Ph.D thesis~\cite[Theorem 3.6]{kim2007}, but never published by the author. After that, the same result was proved again by Futer and Thomas (for $m\ge6$)~\cite[Corollary 1.3]{FT2012} and by Holt and Rees~\cite[Theorem 3.1]{HR2011}. 
\item
Theorem~\ref{thm:cocont} and Corollary~\ref{cor:anticycle} were proved in the author's Ph.D thesis~\cite[Corollaries 4.11 and 4.12]{kim2007}, but never published anywhere. We will give new accounts of these results.
\ee 
The methods presented in this article do not depend on the above mentioned works.
\end{notelit}

\textbf{Acknowledgement.}
I  am grateful to Andrew Casson for his guidance. I thank Fr\'{e}d\'{e}ric Haglund for an inspirational conversation.

\section{Preliminary on cube complexes and label-reading maps}\label{sec:prelim}
\subsection{Local isometries and special cube complexes}
By a \emph{cube complex}, we mean a CW-complex obtained from unit Euclidean cubes of various dimensions by isometrically gluing some of the faces. A \emph{flag complex} is a simplicial complex $L$ such that each complete subgraph in $L^{(1)}$ is the $1$--skeleton of some simplex in $L$. We say a cube complex $X$ is \emph{nonpositively curved}, or simply \emph{NPC}, if the link of each vertex is a flag complex; this is equivalent to saying that the piecewise Euclidean length metric induced on the universal cover of $X$ is CAT(0)~\cite{gromov1987}.

We denote by $ X_\Gam$ the \emph{Salvetti complex} of $A(\Gam)$~\cite{charney2007}.
This means that $ X_\Gam^{(2)}$ is the presentation 2-complex of $A(\Gam)$,
and for each maximal complete subgraph $K$ of $\Gam$ with $k$ vertices, 
a $k$-torus $T$ is glued to $ X_\Gam^{(2)}$ so that the $1$--skeleton of $T$ is the bouquet of the circles corresponding to the vertices of $K$. 
Note that $X_\Gam$ is an NPC cube complex such that $\pi_1(X_\Gam)\cong A(\Gam)$; see~\cite{charney2007}.

If $X$ is a cube complex and $v$ is a vertex of $X$, we denote the link of $X$ at $v$ by $\lk(X;v)$. 
Let us consider a combinatorial map $f:X\to Y$ between cube complexes $X$ and $Y$.
The map $f$ induces a simplicial map $\lk(f;v)$ between $\lk(X;v)$ and $\lk(Y;f(v))$ for each vertex $v$ of $X$.
Following~\cite{charney2000}, we call $f$ as a \emph{local isometry} if 
\begin{enumerate}[(i)]
\item
$\lk(f;v)$ is injective, and
\item
the image of $\lk(f;v)$ is a full subcomplex of $\lk(Y;f(v))$.
\end{enumerate}

\begin{lem}[{\cite{charney2000,CW2004,bh1999}}]\label{npc inj}
Suppose $X$ and $Y$ are cube complexes and $f:X\to Y$ is a combinatorial map.
If $Y$ is NPC and $f$ is a local isometry, then $X$ is also NPC and $f$ is $\pi_1$--injective.
\end{lem}

\begin{defn}[{\cite{HW2008,HW2011}}]\label{defn:special}
\be
\item
A cube complex $X$ is called \emph{special} if $X$ combinatorially maps to a Salvetti complex by a local isometry.
\item
A group $G$ is \emph{special} if $G\cong\pi_1(X)$ for some special cube complex $X$. Furthermore, if $X$ can be chosen to be compact, then we say $G$ is \emph{cocompact special}.
\ee
\end{defn}

We remark that Definition~\ref{defn:special} (1) is different from, but equivalent to, the original definition in~\cite{HW2008}; see \cite[Proposition 3.2]{HW2011}. 
For a group theoretic property $P$, we say a group $G$ is \emph{virtually} $P$ if a finite-index subgroup of $G$ is $P$.
Virtually special groups are of particular interest in 3--manifold theory~\cite{Agol2008}.

\subsection{Label-reading maps}
By a \emph{curve} on a surface, we will mean a simple closed curve or a properly embedded arc.
Let $S$ be a compact surface possibly with boundary.
Suppose $\mV$ is a finite set of transversely intersecting curves on $S$ and $\lambda\co\mV\to V(\Gam)$ is a map such that two curves $\alpha$ and $\beta$ in $\mV$ are intersecting only if $\{\lambda(\alpha),\lambda(\beta)\}\in E(\Gam)$.
Following~\cite{CW2004}, we say that $(\mV,\lambda)$ is a \emph{label-reading pair} on $S$ with the \emph{underlying graph} $\Gam$; and for each $\alpha\in\mV$, we call $\lambda(\alpha)$ as the \emph{label} of $\alpha$. 
If an arc $\alpha$ is labeled by $a\in V(\Gam)$, we say $\alpha$ is an \emph{$a$--arc}.
For each oriented path $\gamma$ transverse to $\mV$, we follow $\gam$ and read off the labels of the curves in $\mV$ that intersect $\gamma$. The word $w(\gamma)$ thus obtained will be called the \emph{label-reading} of $\gamma$ with respect to $(\mV,\lambda)$. The word $w(\gamma)$ represents an element of $C(\Gam)$.
If there exists a group homomorphism $\phi\co \pi_1(S)\to C(\Gam)$ satisfying that $\phi([\gamma])=w(\gamma)$ for each $[\gamma]\in\pi_1(S)$, we call $\phi$ as a \emph{label-reading map} with respect to $(\mV,\lambda)$. 

Recall that a word $w$ representing an element in $C(\Gam)$ is \emph{reduced} if no shorter word represents the same element. It is \emph{cyclically reduced} if $w$ and each of its cyclic conjugations are reduced. 
If a curve $\gamma$ on a compact surface $S$ is homotopic to a subset of $\partial S$ by a homotopy fixing $\partial \gamma$, then we say $\gamma$ is \emph{homotopic into} $\partial S$.

Crisp and Wiest proved that the fundamental group of a closed hyperbolic surface $S$ embeds into some right-angled Artin group if and only if $\chi(S)\ne-1$~\cite{CW2004}. A critical tool for the proof was the realization of an arbitrary group homomorphism $\phi\co \pi_1(S)\to A(\Gam)$ as a label-reading map (using $A(\Gam)$ instead of $C(\Gam)$).
The following is a simple variation of the results in~\cite{CW2004} combined with~\cite{kim2010}.

\begin{thm}[\cite{CW2004,kim2010}]\label{thm:lr}
Let $S$ be a compact surface.
\be
\item
Suppose $(\mV,\lambda)$ is a label-reading pair on $S$ with the underlying graph $\Gam$.
Then for each choice of the base point of $S$, there uniquely exists a label-reading map $\phi\co \pi_1(S)\to C(\Gam)$ with respect to $(\mV,\lambda)$.
\item
Conversely, every group homomorphism $\phi\co\pi_1(S)\to C(\Gam)$ can be realized as a label-reading map with respect to some label-reading pair $(\mV,\lambda)$ that has the underlying graph $\Gam$.
\item
Possibly after composing $\phi$ with an inner automorphism of $C(\Gam)$, we can choose $(\mV,\lambda)$ in (2) further satisfying the following:
\be[(i)]
\item
curves in $\mV$ are minimally intersecting;
\item
curves in $\mV$ are neither null-homotopic nor homotopic into $\partial S$;
\item
for each component $\partial_i S$ of $\partial S$, the label-reading $w(\partial_i S)$ is cyclically reduced.
\ee
\ee
\end{thm}

\begin{proof}
(1) and (2) are proved in~\cite{CW2004} for right-angled Artin groups. The proofs for right-angled Coxeter groups are very similar, except that we now allow $\mV$ to contain orientation-reversing closed curves and also that curves in $\mV$ are \emph{not} assigned with transverse orientations. 
(3) is obtained by lexicographically minimizing the complexity $(|(\cup\mV)\cap\partial S|, |\mV|,\sum_{\alpha\ne\beta\in\mV}|\alpha\cap\beta|)$, possibly after changing $\phi$ by $\mathrm{Inn}(C(\Gam))$; see~\cite{CW2004} and~\cite{kim2010} for discussion on the same technique.
\end{proof}

\section{Special cube complexes for certain subgroups of graph products}\label{sec:cplx}
In this section, we write $V(\Gam)=\{1,2,\ldots,n\}$ and  assume $\mG=\{G_1,G_2,\ldots,G_n\}$ is a collection of groups indexed by $V(\Gam)$.
Choose $k$ such that $|G_i|$ is finite if and only if $1\le i\le k$.
Set $p_i\co GP(\Gam,\mG)\to G_i$ to be the natural projection map for $i=1,2,\ldots,n$.
Recall that we have defined the \emph{graph product kernel} as $KP_0(\Gam,\mG)=\cap_{1\le i\le n} \ker p_i$. We also define $KP_f(\Gam,\mG)=\cap_{1\le i\le k}\ker p_i$. 
Note that $KP_0(\Gam,\mG)\le KP_f(\Gam,\mG)\le GP(\Gam,\mG)$ and that $GP(\Gam,\mG) / KP_f(\Gam,\mG) \cong \prod_{1\le i\le k} G_i$ is finite. 
If all the groups in $\mG$ are abelian, $KP_0(\Gam,\mG)$ is the commutator subgroup of $GP(\Gam,\mG)$.

Let us regard $\mathbb{R}^n$ as a cube complex whose vertices are the lattice points and whose $1$--skeleton consists of the grid lines. We set ${e}_i$ to be the $i$-th standard basis vector.
Following~\cite{SDS1989}, $Y_\Gam \le \mathbb{R}^n$ is defined to be the lift of $X_\Gam\subseteq (S^1)^n$ with respect to the covering $\R^n\to (S^1)^n$. 
Concretely, $(Y_\Gam)^{(1)}=(\R^n)^{(1)}$ and 
 for each complete subgraph of $\Gam$ having the vertex set $\{i_1,\ldots,i_k\}$, 
the following collection of the unit $k$--cubes is contained in $Y_\Gam$:
\[
\left\{ \sum_{j=1}^k t_j {e}_{i_j} \co t_j\in[0,1] \right\}+\mathbb{Z}^n.
\]

We define $Z_0(\Gam,\mG)=Y_\Gam\cap\left(\prod_{i=1}^k[0, | G_i |-1]\times\R^{n-k}\right)\subseteq \R^n$.
We let $Z_\Gam$ denote the preimage of $X_\Gam\subseteq (S^1)^n$ with respect to the covering $\R^k\times (S^1)^{n-k}\to (S^1)^n$ and put $Z_f(\Gam,\mG) = Z_\Gam\cap \left(\prod_{i=1}^k [0,|G_i|-1] \times (S^1)^{n-k}\right)$.
See Figure~\ref{fig:kp}.

\begin{figure}[htb!]
\[
\xymatrix{
Z_0(\Gam,\mG) =Y_\Gam\cap\left(\prod_{i=1}^k[0, | G_i |-1]\times\R^{n-k}\right) \ar@{^(->}[rr]\ar[d] && Y_\Gam\ar@{^(->}[rr]\ar[d]^p && \R^n\ar[d]\\
Z_f(\Gam,\mG) = Z_\Gam\cap \left(\prod_{i=1}^k [0,|G_i|-1] \times (S^1)^{n-k}\right)\ar@{^(->}[rr] && Z_\Gam\ar@{^(->}[rr]\ar[d] && \R^k\times(S^1)^{n-k}\ar[d]\\
&& X_\Gam \ar@{^(->}[rr] &&(S^1)^{n}
}
\]
\caption{Cube complexes in Theorem~\ref{thm:kp}. The horizontal maps are inclusions and the vertical maps are coverings.
\label{fig:kp}}
\end{figure}

It is well-known that 
 $\pi_1(Y_\Gam)=[A(\Gam),A(\Gam)]$ and that $\pi_1(Y_\Gam\cap[0,1]^n)\cong [C(\Gam),C(\Gam)]$; see
 ~\cite{SDS1989,droms2003}. 
We generalize these observations as follows.

\begin{thm}\label{thm:kp}
\be
\item
If $\mG$ consists of countable groups, then $\pi_1(Z_0(\Gam, \mG))\cong KP_0(\Gam,\mG)$ . 
\item
If $\mG$ consists of finite or cyclic groups, then $\pi_1(Z_f(\Gam, \mG))\cong  KP_f(\Gam,\mG)$.
\ee
\end{thm}

\begin{proof}
(1)
We use the notation described so far in this section.
Choose the origin $O\in \R^n$ as the base point.
Let $q_j\co \R^n\to\R$ denotes the natural projection onto the $j$-th component.
Enumerate $G_j = \{g^{(j)}_0=1,g^{(j)}_1,g^{(j)}_2,\ldots, g^{(j)}_{|G_j|-1}\}$ for each $j\le k$, and 
$G_j = \{\ldots,g^{(j)}_{-1},g^{(j)}_{0}=1,g^{(j)}_{1},\ldots\}$ for each $j>k$.
We let $\mH$ be the homotopy classes of the edge-paths in $Z_0(\Gam,\mG)$ starting from $O$.

We will first define a map $\Phi\co\mH\to GP(\Gam,\mG)$.
Let us consider an edge-path $\gamma\co[0,l]\to Z_0(\Gam, \mG)$ such that $\gamma(0)=O$ and $\gamma^{-1}(\Z^n\cap Z_0(\Gam,\mG))=\Z\cap[0,l]$.
For each $i=1,2,\ldots,l$, there uniquely exists $j$ such that $\gamma[i,i+1]$ is parallel to $\R e_j$;
then we put $k=q_j\circ\gamma(i), k'=q_j\circ\gamma(i+1)$ and $x_i = (g^{(j)}_k)^{-1}g^{(j)}_{k'} \in G_j$.
Note that $|k-k'|=1$.
If $\gamma[i,i+1]$ and $\gamma[i+1,i+2]$ span a $2$-cell in $Z_0(\Gam, \mG)$, then the groups containing $x_i$ and $x_{i+1}$ commute; that is, $x_i x_{i+1} = x_{i+1} x_i$. 
So we can define a map $\Phi\co \mH \to GP(\Gam,\mG)$ by setting $\Phi([\gamma])=x_1 x_2\cdots x_l$.

Conversely, suppose $1\ne g\in GP(\Gam,\mG)$ is given. 
The normal form theorem for graph products~\cite{green1990,HW1999} implies that we can write  $g=g^{(j_1)}_{k_1} g^{(j_2)}_{k_2}   \cdots g^{(j_l)}_{k_l}$ such that:
\be[(i)]
\item
$k_i\ne0$ for each $i=1,2,\ldots,l$; 
\item
$j_i\ne j_{i+1}$ for each $i=1,2,\ldots,l-1$;
\item
if $j_i = j_{i'}$ for some $i<i'$, then there exists $i<i''<i'$ such that $\{j_i,j_{i''}\}=\{j_{i'},j_{i''}\}\not\in E(\Gam)$.
\setcounter{saveenum}{\value{enumi}}
\ee
Let us fix $i\in\{1,2,\ldots,l\}$ and put $j=j_i$. There exist $k$ and $k'$ such that
\[ g^{(j)}_k = \prod_{1\le t<i\textrm{ and }j_{t}=j} g^{(j)}_{k_t}\quad \mbox{ and }\quad
g^{(j)}_{k'}= \prod_{1\le t\le i\textrm{ and }j_{t}=j} g^{(j)}_{k_t} = g^{(j)}_k g^{(j)}_{k_i}.\]
We inductively define $\gamma_i$ to be the edge-path in $\R^n$ 
starting from the end point of $\gamma_{i-1}$ and changing only its $j$-th coordinate from $k$ to $k'$. We set $O$ as the initial point of $\gamma_1$.
By defining $\Psi(g)=[\gamma_1 \gamma_2 \cdots \gamma_l]$, we have a map $\Psi\co GP(\Gam,\mG)\to \mH$.
Note that $\Psi$ is well-defined since two normal forms differ only by a finite sequence of swapping certain consecutive terms, which can also be realized as a homotopy in $Z_0(\Gam,\mG)$. It is clear that $\Psi$ is the (set-theoretic) inverse of $\Phi$.

Now if $1\ne g\in KP_0(\Gam,\mG)$, then we further have:
\be[(i)]
\setcounter{enumi}{\value{saveenum}}
\item $\prod_{1\le t\le l\textrm{ and }j_t = j} g^{(j)}_{k_t}$ is trivial in $G_j$ for each $j=1,2,\ldots,n$.
\ee
It is clear that $\Psi$ restricts to a group isomorphism from $KP_0(\Gam,\mG)$ onto $\pi_1(Z_0(\Gam,\mG))$.

(2) 
In the case when $1\ne g\in KP_f(\Gam,\mG)$, we have the following condition instead of (iv) above:
\be[(i)]
\setcounter{enumi}{\value{saveenum}}
\item[(iv)'] $\prod_{1\le t\le l\textrm{ and }j_t = j} g^{(j)}_{k_t}$ is trivial in $G_j$ for each $j=1,2,\ldots,k$.
\ee
Hence, the covering $Y_\Gam\to Z_\Gam$ projects $\Psi(g)$ to a homotopy class of a loop in $Z_f(\Gam,\mG)$.
We can check that the restriction of $\psi$ onto $KP_f(\Gam,\mG)$ is a group isomorphism onto $\pi_1(Z_f(\Gam,\mG))$.
\end{proof}

\begin{rem}\label{rem:kp}
The $\mH$ in the above proof depends only on the orders of the groups in $\mG$. Hence, $\Phi$ determines a bijection between $GP(\Gam,\{\Z_{|G_i|}\co i=1,2,\ldots,n\})$ and $GP(\Gam,\mG)$.
In particular, if all the groups in $\mG$ are infinite, then $\Phi$ induces a bijection between $A(\Gam)$ and $GP(\Gam,\mG)$.
\end{rem}

\begin{exmp}\label{exmp:kp}
Let us consider $G =(\Z_3\ast \Z_4)\times\Z$, which is regarded as the graph product in Figure~\ref{fig:exmp:kp} (a).
Then $KP_0(\Gam,\mG)=[G,G]$ since the vertex groups are abelian. Also, $KP_f(\Gam,\mG)=[\Z_3\ast \Z_4,\Z_3\ast \Z_4]\times \Z$ is a subgroup of $G$ with index $|\Z_3|\cdot|\Z_4| = 12$. 
If $\Lambda$ is the graph shown in Figure~\ref{fig:exmp:kp} (b), 
then we have $Z_0(\Gam,\mG)=\Lambda\times\R$ and  $Z_f(\Gam,\mG)=\Lambda\times S^1$.
\end{exmp}

\begin{figure}[htb!]
  \tikzstyle {a}=[red,postaction=decorate,decoration={%
    markings,%
    mark=at position 1 with {\arrow[red]{stealth};}}]
  \tikzstyle {b}=[postaction=decorate,decoration={%
    markings,%
    mark=at position .57 with {\arrow{stealth};},%
    }]
  \tikzstyle {v}=[draw,shape=circle,fill=black,inner sep=0pt]
  \tikzstyle {bv}=[black,draw,shape=circle,fill=black,inner sep=1pt]
  \tikzstyle{every edge}=[-,draw]
\subfloat[(a) $G$]{
	\begin{tikzpicture}[thick]
    	  \draw[]  (-2,0) node  [bv]  (am) {} node[below] {$\Z_3$} 
    	  --  (0,0) node [bv] (a0) {} node[below] {$\Z$}
	  --  (2,0) node [bv] (ap) {} node[below] {$\Z_4$};
	  \draw (0,-.7) node [] {};
	\end{tikzpicture}
	}
	\qquad\qquad\qquad\qquad
	\subfloat[(b) $\Lambda$]{
\begin{tikzpicture}[thick]
	\draw (-1.5,-1) -- (1.5,-1) -- (1.5,1) -- (-1.5,1) -- cycle;
	\draw (-.5,-1) -- (-.5,1);
	\draw (.5,-1) -- (.5,1);	
	\draw (-1.5,0) -- (1.5,0);
\end{tikzpicture}
	}
\caption{Example~\ref{exmp:kp}.\label{fig:exmp:kp}}
\end{figure}

\begin{proof}[Proof of Theorem~\ref{thm:v special}]
Let us use the notations in the proof of Theorem~\ref{thm:kp}.
Note that the compositions $Z_0(\Gam,\mG)\hookrightarrow Y_\Gam\to X_\Gam$ and $Z_f(\Gam,\mG)\hookrightarrow Z_\Gam\to X_\Gam$ are local isometries,
and that $Z_f(\Gam,\mG)$ is compact. Moreover, $[GP(\Gam,\mG):KP_f(\Gam,\mG)]=\prod_{i=1}^k |G_i|$ is finite.
\end{proof}

If $C\subseteq D$ are (possibly infinite) rectangular boxes in $\mathbb{R}^n$ whose vertices are lattice points,
then the inclusion $Y_\Gam\cap C\hookrightarrow Y_\Gam\cap D$ is a local isometry.
We note two immediate corollaries of Theorem~\ref{thm:kp}.

\begin{cor}\label{cor:gpk}
\be 
\item
The graph product kernel of countable groups embeds into $[A(\Gam),A(\Gam)]$.
\item
The graph product of finite or cyclic groups virtually embeds into $A(\Gam)$.
\ee\end{cor}

\begin{cor}\label{cor:inclusion}
Let $\mG=\{G_v\co v\in V(\Gam)\}$
and $\mG'=\{G'_v\co v\in V(\Gam)\}$ be collection of countable groups.
\begin{enumerate}
\item
If $|G_v|\le |G'_v|\le\infty$ for each vertex $v$, then $KP_0(\Gam,\mG)$ embeds into $KP_0(\Gam,\mG' )$.
If we further assume that $\mG$ and $\mG'$ consist of finite or cyclic groups,
then $KP_f(\Gam,\mG)$ embeds into $KP_f(\Gam,\mG' )$.
\item
If $|G_v| = |G'_v|\le\infty$ for each vertex $v$, then $KP_0(\Gam,\mG)\cong KP_0(\Gam,\mG' )$.
\end{enumerate}
\end{cor}

\begin{lem}\label{lem:inclusion}
Each finitely generated subgroup of $[A(\Gam),A(\Gam)]$ embeds into $[GP_m(\Gam),GP_m(\Gam)]$ for some $2\le m < \infty$.
\end{lem}

\begin{proof}
Let $H$ be a finitely generated subgroup of $[A(\Gam),A(\Gam)]$ and $B$ be the union of edge-paths in $Y_\Gam$ that correspond to the generators of $H$. For $n=|V(\Gam)|$, there exists $2\le m<\infty$ such that $B$ is contained in $[0,m-1]^n$.
Then $H\to [A(\Gam),A(\Gam)]$ factors as $H\to \pi_1(Y_\Gam\cap[0,m-1]^n)=[GP_m(\Gam),GP_m(\Gam)]\to \pi_1(Y_\Gam)= [A(\Gam),A(\Gam)]$.
\end{proof}

\begin{proof}[Proof of Theorem~\ref{thm:artin coxeter}]
For each of (1) and (2), only one direction of the assertion is not obvious. We will follow the notations in the proof of Theorem~\ref{thm:kp}.

(1)
Suppose $C(\Gam)\in\mS$.
If $\mG$ consists of nontrivial groups,
Corollary~\ref{cor:inclusion} implies that $[C(\Gam),C(\Gam)]$ embeds into $KP_0(\Gam,\mG)$. Note that $[C(\Gam),C(\Gam)]$ is in $\mS$ since $[C(\Gam)\co [C(\Gam),C(\Gam)]]<\infty$.

(2) Suppose $\mG$ consists of cyclic groups and $GP(\Gam,\mG)\in\mS$.
Then $KP_f(\Gam,\mG)$ is in $\mS$ since it is a finite-index subgroup of $GP(\Gam,\mG)$.
The conclusion follows since $KP_f(\Gam,\mG)\le\pi_1(Z_\Gam)\le A(\Gam)$. 

(3) By Lemma~\ref{lem:inclusion}, $[A(\Gam),A(\Gam)]\in\mS$ if and only if $GP_m(\Gam)\in\mS$ for sufficiently large $m$. Also note that if $\mG$ consists of finite groups, then $KP_0(\Gam,\mG)$ embeds into $GP_m(\Gam)$ for sufficiently large $m$.
\end{proof}

\begin{proof}[Proof of Corollary~\ref{cor:cycle} (1)]
Since $C(C_m)$ is a cocompact Fuchsian group, it contains a hyperbolic surface group.
Apply Theorem~\ref{thm:artin coxeter} (1).
\end{proof}

\section{Doubles and co-contractions}\label{sec:double} 
Suppose $A$ and $B$ are groups.
For an isomorphism $\psi\co C\to D$ where $C\le A$ and $D\le B$, we let $A\ast_\psi B$ denote the free product of $A$ and $B$ amalgamated along $\psi$.
If $\psi\co C\to C'$ is an isomorphism for some $C,C'\le A$, then the HNN extension of $A$ along $\psi$ is denoted as $A\ast_\psi$.
\begin{lem}\label{lem:phi}
Suppose $G$ is a group, $\psi\co H\to H$ is an isomorphism for some $H\le G$ and $2\le k<\infty$.
Let $G_k = G\ast_\psi\cdots\ast_\psi G$ where there are $k$ copies of $G$.
We denote the stable generator of $G\ast_\psi$ by $t$.
\begin{enumerate}[(1)]
\item The group $G_k$ embeds into $G\ast_{\psi}/\fform{t^k}$ as a subgroup of index $k$.
\item The group $G\ast_{\psi}/\fform{t^k}$ virtually embeds into $G\ast_\psi$.
\end{enumerate}
\end{lem}

\begin{proof}
Let $L_k = G_k\ast_\psi$, whose stable generator is denoted by $s$.
The groups in the first line of Figure~\ref{fig:double} are illustrated as graphs of groups, where each vertex corresponds to $G$ and each directed edge corresponds to $\psi\co H\to H$; (b) shows the case when $k=5$ as an example.
The number $1\times k$ in (e) means that $G\ast_\psi/\fform{t^k}$ is obtained from $G\ast_\psi$ by gluing (to a classifying space of $G\ast_\psi$) $D^2$ such that $[\partial D^2]$ and $t^k$ are identified.
Similarly, $G_k \cong L_k/\fform{s}$ in (d) is obtained from $L_k$ by attaching $k$ copies of $D^2$, whose boundary curves are all identified with $s$; this is described as $k\times 1$.
The figure shows a commutative diagram, where $p_1$ and $p_2$ are induced by covering maps.
Now for (1), note that $p_2$ is injective.
(2) follows from that $G_k$ embeds into $L_k$ and that $p_1$ is injective.
\end{proof}

The following is a special case of Lemma~\ref{lem:phi} (2).

\begin{cor}\label{cor:v embed}
Let $G$ and $K$ be graph products of groups such that $K$ is obtained from $G$ by replacing a finite cyclic vertex group of $G$ by $\Z$. Then $G$ virtually embeds into $K$.
\end{cor}

\begin{exmp}\label{exmp:v embed}
\be
\item
Consider the graph product $G$ shown in Figure~\ref{fig:exmp:kp} (a).
Suppose $K$ is obtained from $G$ by substituting $\Z$ for $\Z_4$. Then $G$ virtually embeds into $K$.
\item
We can apply Lemma~\ref{lem:phi} (2) repeatedly to see that $C(\Gam)$ virtually embeds into $A(\Gam)$.
This also follows from the well-known fact that $[C(\Gam),C(\Gam)]\le [A(\Gam),A(\Gam)]$; see~\cite{droms2003}. 
\ee
\end{exmp}

We let $G\ast_H G$ denote the free product of two copies of $G$ amalgamated along $id_H$. 
The HNN extension of $G$ along $id_H$ is denoted as $G\ast_H$.

\begin{lem}\label{lem:double}
Let $G$ be a group, $H\le G$ and $k\in\Z\ssm\{1,-1\}$.
Then $G\ast_H G$ embeds into $G\ast_H/ \fform{t^k}$.
\end{lem}

\begin{proof}
The case when $k\ne0$ follows from Lemma~\ref{lem:phi} (1). 
For $k=0$, consider the first line of Figure~\ref{fig:double}.
We remark that $k=0$ case is already well-known; see \cite[p.187]{LS2001}.
\end{proof}

\begin{figure}[htb!]
  \tikzstyle {a}=[red,postaction=decorate,decoration={%
    markings,%
    mark=at position 1 with {\arrow[red]{stealth};}}]
  \tikzstyle {b}=[postaction=decorate,decoration={%
    markings,%
    mark=at position .57 with {\arrow{stealth};},%
    }]
  \tikzstyle {v}=[draw,shape=circle,fill=black,inner sep=0pt]
  \tikzstyle {bv}=[black,draw,shape=circle,fill=black,inner sep=1pt]
  \tikzstyle{every edge}=[-,draw]
\xymatrix{
	\subfloat[(a) $G\ast_\psi G$]{
	\begin{tikzpicture}[thick]
    	  \draw[b]  (-1,0) node  [bv]  (am) {} node[below] {} 
    	  -- node[midway,above] {} (1,0) node [bv] (ap) {} node[below] {};
	  \draw (0,.5) node [] {$\xleftarrow{id_H}  {\scriptstyle H} \xrightarrow{\psi}$};
	  \draw (0,-.5) node [] {};
	\end{tikzpicture}
	}
	\qquad
	\ar@<4ex>@{^(->}^{\textrm{incl}_*}[rr]
	&&
	\qquad
	\subfloat[(b) $L_k$]{
	\begin{tikzpicture}[thick]
   		\foreach \i in {0,...,4} 
			\draw (360/5*\i+90:1) node [bv] (v\i) {} ;
		\draw[b] (v1)--(v2);
		\draw[b] (v2)--(v3);
		\draw[b] (v4)--(v0);
		\draw[b] (v0)--(v1);
		\draw[b] (v4)--(v3);
	\end{tikzpicture}
	}
	\qquad
	\ar@<4ex>@{^(->}^---{p_1}_---{s\;\mapsto\; t^k}[rr]
	\ar@<-1ex>[d]_--{\textrm{incl}_*} 
	&&
	\qquad
  	\subfloat[(c) $G\ast_\psi$]{
  	\begin{tikzpicture}[thick]
		\draw (0,0) node[bv] {};
		\draw (0,-.4) node[] {};
		\draw [out=30,in=90] (0,0) edge  (2,0);
		\draw [out=30,in=90] (0,0) edge  (2,0);					
		\draw (2.2,0) node [right] {}; 
		\draw [out=-30,in=-90] (0,0) edge  (2,0);	
		\draw [b] (2,-.1) -- (2,.1);				
	\end{tikzpicture}
	}
	\ar^--{\textrm{incl}_*}[d]
	\\
	\qquad
	&&
	\qquad
	\subfloat[(d) $G_k\cong L_k/\fform{s}$]{
	 	\begin{tikzpicture}[thick]
   		\foreach \i in {0,...,4} 
			\draw (360/5*\i+90:1) node [bv] (v\i) {} ;
		\draw[b] (v1)--(v2);
		\draw[b] (v2)--(v3);
		\draw[b] (v4)--(v0);
		\draw[b] (v0)--(v1);
		\draw[b] (v4)--(v3);
		\draw (0,0) node [] {$\scriptstyle\textcolor{red}{k}\times \textcolor{blue}{1}$};
	 	\end{tikzpicture}
	}
	\qquad
	\ar@<4ex>@{^(->}^---{p_2}[rr]\ar@<-1ex>@{_(->}[u]
	&&
	\qquad
	\subfloat[(e) $G\ast_\psi/\fform{t^k}$]{
  	\begin{tikzpicture}[thick]
		\draw (0,0) node[bv] {};
		\draw (0,-.4) node[] {$\scriptstyle G$};
		\draw [out=30,in=90] (0,0) edge  (2,0);
		\draw [out=30,in=90] (0,0) edge  (2,0);					
		\draw [out=-30,in=-90] (0,0) edge  (2,0);	
		\draw (1.2,0) node [] {$\scriptstyle \textcolor{red}{1}\times \textcolor{blue}{k}$};
		\draw [b] (2,-.1) -- (2,.1);						
	\end{tikzpicture}
	}
}
\caption{Lemma~\ref{lem:phi} and~\ref{lem:double}.\label{fig:double}}
\end{figure}

Let us consider a graph isomorphism $\mu \co \Gam \to \Gam'$. 
Fix a vertex $t\in V(\Gam)$ and put $\Lambda$ the subgraph of $\Gam$ induced by $\lk(t)$.
Define $\Gam''$ as the graph obtained from the union of $\Gam\ssm \st(t)$ and $\Gam'\ssm \st(\mu(t))$ after identifying $\Lambda$ and $\mu(\Lambda)$ by $\mu$.
Here, $\st(v)$ denotes the open star of a vertex $v$.
We call $\Gam''$ as the \emph{double of $\Gam\ssm\st(t)$ along the link of $t$}.
There is a natural projection map $\rho\co \Gam''\to \Gam\ssm\st(t)$.

\begin{lem}\label{lem:double}
If $\mG = \{ G_v \co v\in V(\Gam)\}$ is a collection of nontrivial groups and $\Gam''$ is the double of $\Gam\ssm\st(t)$ along the link of $t$,
then there is an embedding from
$GP(\Gam'',\{G_{\rho(u)}\co u\in V(\Gam'')\})$ into $GP(\Gam,\mG)$.
\end{lem}

\begin{proof}
Since $G_t$ contains a nontrivial cyclic subgroup, we may assume $G_t=\Z/k\Z$ for some $2\le k<\infty$ or $k=0$. Put $G = GP(\Gam\ssm\st(t),\mG\ssm \{G_t\})$ and $H=\form{\{G_v\co v\in \lk(t)\}}\le G$. Then we have $GP(\Gam,\mG) = G\ast_H /\fform{t^k}$ and  Lemma~\ref{lem:double} applies.
\end{proof}

For $e\in E(\Gam)$, we let $\Gam/e$ denote the contraction of $e$~\cite[p.20]{diestel2010}.

\begin{cor}\label{cor:cocont}
Suppose $e=\{x,t\}$ is an edge of $\Gam\opp$.
We put $\Gam_0\opp = \Gam\opp / e$ and denote by $y\in V(\Gam_0)=V(\Gam_0\opp)$ the contracted vertex of $e$.
If $\mG = \{G_v\co v\in V(\Gam)\}$ is a collection of nontrivial groups and $G_y = G_x$, then $GP(\Gam_0,(\mG\ssm\{G_x,G_t\})\cup \{G_y\})$ embeds into $GP(\Gam,\mG)$.
\end{cor}

\begin{proof}
As in Figure~\ref{fig:cocont} (a), we partition $V(\Gam)\ssm\{x,t\}$ into $P,Q,R,S$ where $P = \lk(x)\cap\lk(t), Q=\lk(t)\ssm\lk(x),R=\lk(x)\ssm\lk(t)$ and $S=V(\Gam)\ssm(P\cup Q\cup R\cup\{x,t\})$. Let $\Gam_1$ be the double of $\Gam\ssm\st(t)$ along the link of $t$ as shown in Figure~\ref{fig:cocont} (c); we write $R',S'$ and $x'$ for the copies of $R,S$ and $x$. Figure~\ref{fig:cocont} (b) shows that $\Gam_0$ is isomorphic to the subgraph of $\Gam_1$ induced by $P\cup Q\cup R\cup S\cup \{x'\}$.\end{proof}

\begin{figure}[htb!]
\tikzstyle {a}=[red,postaction=decorate,decoration={%
markings,%
mark=at position .5 with {\arrow[red]{stealth};}}]
\tikzstyle {b}=[blue,postaction=decorate,decoration={%
markings,%
mark=at position .43 with {\arrow[blue]{stealth};},%
mark=at position .57 with {\arrow[blue]{stealth};}}]
\tikzstyle {v}=[draw,shape=circle,fill=black,inner sep=0pt]
\tikzstyle {bv}=[blue,draw,shape=circle,fill=blue,inner sep=1pt]
\tikzstyle {gv}=[green,draw,shape=circle,fill=green,inner sep=1pt]
\tikzstyle {rv}=[red,draw,shape=circle,fill=red,inner sep=1pt]
\tikzstyle {blv}=[black,draw,shape=circle,fill=black,inner sep=1pt]
\tikzstyle{every edge}=[-,draw]
\subfloat[(a) $\Gam$]{
\begin{tikzpicture}[thick]
	\draw (-1.5,-1) -- (1.5,-1) -- (1.5,1) -- (-1.5,1) -- cycle;
	\draw (-1.5,0) -- (1.5,0);
	\draw (0,-1) -- (0,1);
	\draw (0.75,-0.5) node {$\scriptstyle P$};
	\draw (0.75,0.5) node {$\scriptstyle Q$};
	\draw (-0.75,0.5) node {$\scriptstyle S$};
	\draw (-0.75,-0.5) node {$\scriptstyle R$};
	\draw (1.2,0.35) node [blv] {} -- (2,0) node [] {};
	\draw (1.3,-0.35) node [blv] {} -- (2,0) node [] {};
	\draw (1.3,-0.7) node [blv] {} -- (2,0) node [] {};
	\draw (1.3,-0.7) node [] {} -- (0,-1.5) node [] {};
	\draw (-.3,-.7) node [blv] {} -- (0,-1.5) node [] {};
	\draw (1.3,-0.35) node [] {} -- (0,-1.5) node [bv] {} node [left] {$\scriptstyle x$};
	\draw (-2,0) node [] {};
	\draw (1.2,0.7) node [blv] {} -- (2,0) node [right] {$\scriptstyle t$};	
	\draw (2,0) node [rv] {};
\end{tikzpicture}
}%
\qquad\qquad
\subfloat[(b) $\Gam_0$]{
\begin{tikzpicture}[thick]
	\draw (-1.5,-1) -- (1.5,-1) -- (1.5,1) -- (-1.5,1) -- cycle;
	\draw (-1.5,0) -- (1.5,0);
	\draw (0,-1) -- (0,1);
	\draw (0.75,-0.5) node {$\scriptstyle P$};
	\draw (0.75,0.5) node {$\scriptstyle Q$};
	\draw (-0.75,0.5) node {$\scriptstyle S$};
	\draw (-0.75,-0.5) node {$\scriptstyle R$};
	\draw (1.3,-0.35) node [blv] {} -- (2,0) node [right] {$\scriptstyle y$};
	\draw (1.3,-0.7) node [blv] {} -- (2,0) node [] {};
	\draw (-.3,-.7) node [] {}  (0,-1.5) node [] {};	
	\draw (-2,0) node [] {};
	\draw	(2,0) node [bv] {};
\end{tikzpicture}
}%
\\
\subfloat[(c) $\Gam_1$]{
\begin{tikzpicture}[thick]
	\draw (-1.5,-1) -- (3,-1) -- (3,1) -- (-1.5,1) -- cycle;
	\draw (-1.5,0) -- (3,0);
	\draw (0,-1) -- (0,1);
	\draw (1.5,-1) -- (1.5,1);	
	\draw (0.75,-0.5) node {$\scriptstyle P$};
	\draw (0.75,0.5) node {$\scriptstyle Q$};
	\draw (-0.75,0.5) node {$\scriptstyle S$};
	\draw (-0.75,-0.5) node {$\scriptstyle R$};
	\draw (2.25,0.5) node {$\scriptstyle S'$};
	\draw (2.25,-0.5) node {$\scriptstyle R'$};
	\draw (1.3,-0.35) node [blv] {};
	\draw (1.3,-0.7) node [blv] {};
	\draw (1.3,-0.35) node [] {} -- (0,-1.5) node [left] {$\scriptstyle x$};
	\draw (1.3,-0.7) node [] {} -- (0,-1.5) node [] {};
	\draw (-.3,-.7) node [blv] {} -- (0,-1.5) node [] {};
	\draw (1.3,-0.35) node [] {} -- (2.5,-1.5) node [right] {$\scriptstyle x'$};
	\draw (1.3,-0.7) node [] {} -- (2.5,-1.5) node [] {};
	\draw (2.7,-.7) node [blv] {} -- (2.5,-1.5) node [] {};
	\draw (0,-1.5) node [bv] {};
	\draw (2.5,-1.5) node [bv] {};
\end{tikzpicture}
}%
\caption{Co-contraction is contained in the double along a link.}
\label{fig:cocont}
\end{figure}
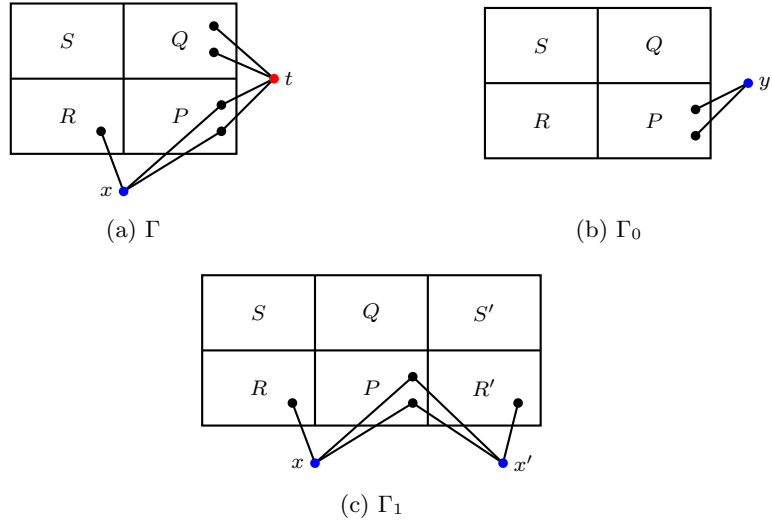

Theorem~\ref{thm:cocont} now follows from Corollary \ref{cor:cocont}. Since there is a finite sequence of edge-contractions from $C_m$ onto $C_5$ for each $m\ge5$, we have an embedding from $C(C_5)\cong C(C_5\opp)$ into $C(C_m\opp)$. Hence, Corollary~\ref{cor:anticycle} also follows.

\begin{rem}\label{rem:cocont}
A special case of Corollary~\ref{cor:cocont} is when all the vertex groups are infinite cyclic, and was proved in~\cite{kim2008}.
Bell gave a shorter proof of this case using the same decomposition of a graph as shown in Figure~\ref{fig:cocont}~\cite{bell2012}, independently from this writing.
\end{rem}

\section{Right-angled Coxeter groups on seven vertex graphs}\label{sec:7v}

For a group $G$ and $H\le G$,
we say that $g\in G$ is \emph{conjugate into $H$} if $g$ is conjugate to an element of $H$.

\begin{defn}\label{defn:small}
Let a group $G$ and its subgroup $H$ be given.
We say $(G,H)$ is \emph{big} if there exists a compact hyperbolic surface $S$ and a monomorphism $\phi\co\pi_1(S)\to G$ such that $\phi([\gamma])$ is conjugate into $H$ whenever $\gamma$ is homotopic into $\partial S$.
The pair $(G,H)$ will be called \emph{small} if it is not big.
\end{defn}

\begin{rem}\label{rem:basic small}
\be
\item
If $(G,H)$ is big and $H\le K\le G$, then $(G,K)$ is also big.
\item
If $G\in\mS$, then $(G,H)$ is big for each $H\le G$.
\item
If a group $G$ does not contain $F_2$, then $(G,H)$ is small for each $H\le G$.
\ee
\end{rem}

The following lemma is obvious from typical transversality arguments.

\begin{lem}[\cite{kim2010}]\label{lem:transverse}
\be
\item
Let $A$ and $B$ be groups and $\psi\co C\to D$ be an isomorphism where $C\le A$ and $D\le B$. 
If $A\ast_\psi B\in \mS$, then either $(A,C)$ or $(B,D)$ is big.
\item
Let $A$ be a group and $\psi\co C\to C'$ be an isomorphism where $C,C'\le A$.
If $A\ast_\psi\in \mS$, then $(A,\form{C,C'})$ is big.
\ee
\end{lem}

\begin{lem}\label{lem:links}
Suppose $a,b\in V(\Gam)$ and $\lk(b)\subseteq \lk(a)$.
If $C(\Gam)\not\in\mS$, then $(C(\Gam),\form{a,b})$ is small.
\end{lem}

\begin{proof}
Suppose $S$ is a compact hyperbolic surface and
$\phi\co\pi_1(S)\to C(\Gam)$ is a monomorphism 
such that $\phi([\gamma])$ is conjugate into $\form{a,b}$ 
whenever $\gamma$ is homotopic into $\partial S$.
Since $C(\Gam)\not\in\mS$, we have $\partial S\ne\varnothing$.
Let $\partial_1 S,\partial_2 S,\ldots$ be the boundary components of $S$. 
Since $\form{a,b}$ contains $\Z$, we see that $a$ and $b$ are distinct and non-adjacent in $\Gam$.
We realize $\phi$ as a label-reading map with respect to a label-reading pair $(\mV,\lambda)$ satisfying the three conditions in Theorem~\ref{thm:lr} (3).
Then each $\partial_i S$ intersects with both $a$--arcs and $b$--arcs, and no arcs with labels other than $a$ or $b$ intersect $\partial S$.
Choose a $b$--arc $\beta$ joining say, $\partial_i S$ and $\partial_j S$. 
These two boundary components may coincide; 
however, $\beta$ is never homotopic into $\partial S$. With suitable choices of the base point and the orientations, $\phi([\beta\cdot\partial_j S\cdot \beta^{-1}]) = w(ab)^l w^{-1}$ for some $l\ne0$ where $w$ is the label-reading of $\beta$. 
Since $w\in\form{\lk(b)} = \form{\lk(a)\cap\lk(b)}$, we have $\phi([\beta\cdot\partial_j S\cdot \beta^{-1}]) = (ab)^l$. Also, $\phi([\partial_i S]) = (ab)^{m}$ for some $m\ne0$. Since $[\beta\cdot\partial_j S \cdot\beta^{-1}]$ and $[\partial_i S]$ do not commute, we have a contradiction.
\end{proof}

\begin{lem}\label{lem:product}
Let $H,G_1,G_2$ be groups such that $H$ is torsion-free word-hyperbolic.
Denote by $p_i\co G_1\times G_2\to G_i$ the natural projection for $i=1,2$.
If $\phi\co H\to G_1\times G_2$ is injective, then $p_1\circ\phi$ or $p_2\circ\phi$ is also injective.
\end{lem}
\begin{proof}
Suppose $1\ne x_1\in \ker (p_2\circ\phi)$ and $1\ne x_2\in \ker (p_1\circ\phi)$ so that $\phi(x_i)\in G_i$ for $i=1,2$.
Since $\phi([x_1,x_2])=[\phi(x_1),\phi(x_2)]=1$, we have $x_1^M=x_2^N$ for some $M,N\ne0$~\cite[Corollary 3.10]{bh1999}.
So, $\phi(x_1^M)=\phi(x_2^N)\in G_1\cap G_2=1$.
\end{proof}

A repeated application of Lemma~\ref{lem:product} easily implies the following.

\begin{lem}\label{lem:small product}
If $k>0$ and $(G_i,H_i)$ is small for $i=1,2,\ldots,k$,
then $(\prod_{i=1}^k G_i, \prod_{i=1}^k H_i)$ is small.
\end{lem}

\begin{exmp}\label{exmp:p3}
Label the vertices of $P_3\coprod P_3$ as Figure~\ref{fig:7v} (a) and let $\Lambda_0 = (P_3\coprod P_3)\opp$. 
Let us consider subgroups of $C(\Lambda_0)=\form{a,b,c,e,f,g}$ from now on.
In the subgraph of $\Lambda_0$ induced by $\{a,b,c\}$, we have $\lk(b)=\varnothing$ and $\lk(a)=c$.
Lemma~\ref{lem:links} implies that $(\form{a,b,c},\form{a,b})$ is small.
By Lemma~\ref{lem:small product}, $(\form{a,b,c,e,f,g},\form{a,b,f,g})=(\form{a,b,c}\times\form{e,f,g},\form{a,b}\times\form{f,g})$ is also small.
\end{exmp}

\begin{figure}[htb!]
  \tikzstyle {b}=[postaction=decorate,decoration={%
    markings,%
    mark=at position .57 with {\arrow{stealth};},%
    }]
  \tikzstyle {v}=[draw,shape=circle,fill=black,inner sep=0pt]
  \tikzstyle {bv}=[black,draw,shape=circle,fill=black,inner sep=1pt]
  \tikzstyle{every edge}=[-,draw]
	\subfloat[(a) $\Lambda_0\opp=P_3\coprod P_3$]{
	\begin{tikzpicture}[thick]
    	  \draw  (-1.3,0) node  [bv]  (am) {} node[above] {$\scriptstyle a$} 
    	  -- (-.8,0) node [bv] (ap) {} node[above] {$\scriptstyle b$}
   	  -- (-.3,0) node [bv] (ap) {} node[above] {$\scriptstyle c$};
    	  \draw  (.3,0) node  [bv]  (am) {} node[above] {$\scriptstyle e$} 
    	  -- (.8,0) node [bv] (ap) {} node[above] {$\scriptstyle f$}
   	  -- (1.3,0) node [bv] (ap) {} node[above] {$\scriptstyle g$};	  
	  \draw (0,-.6) node [] {};
	\end{tikzpicture}
	}
	\qquad
	\qquad
	\subfloat[(b) $P_6$]{
	\begin{tikzpicture}[thick]
    	  \draw  (-1.5,0) node  [bv]  (am) {} node[above] {$\scriptstyle a$} 
    	  -- (-1,0) node [bv] (ap) {} node[above] {$\scriptstyle b$}
   	  -- (-.5,0) node [bv] (ap) {} node[above] {$\scriptstyle c$}
   	  -- (0,0) node [bv] (ap) {} node[above] {$\scriptstyle d$}	  
   	  -- (.5,0) node [bv] (ap) {} node[above] {$\scriptstyle e$}	  
   	  -- (1,0) node [bv] (ap) {} node[above] {$\scriptstyle f$};
	  \draw (0,-.6) node [] {};	  
	\end{tikzpicture}
	}
	\qquad
	\qquad
	\subfloat[(c) $P_7$]{
	\begin{tikzpicture}[thick]
    	  \draw  (-1.5,0) node  [bv]  (am) {} node[above] {$\scriptstyle a$} 
    	  -- (-1,0) node [bv] (ap) {} node[above] {$\scriptstyle b$}
   	  -- (-.5,0) node [bv] (ap) {} node[above] {$\scriptstyle c$}
   	  -- (0,0) node [bv] (ap) {} node[above] {$\scriptstyle d$}	  
   	  -- (.5,0) node [bv] (ap) {} node[above] {$\scriptstyle e$}	  
   	  -- (1,0) node [bv] (ap) {} node[above] {$\scriptstyle f$}
   	  -- (1.5,0) node [bv] (ap) {} node[above] {$\scriptstyle g$};
	  \draw (0,-.6) node [] {};	  
	\end{tikzpicture}
	}	
\caption{Some six and seven vertex graphs}
\label{fig:7v}
\end{figure}
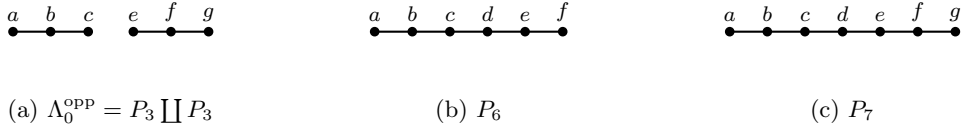

\begin{lem}\label{lem:finite amalgam}
If the free product of two groups $A$ and $B$ amalgamated along a finite subgroup is in $\mS$, then either $A$ or $B$ is in $\mS$.
\end{lem}

\begin{proof}
Suppose $H$ is a hyperbolic surface group and $C$ is a finite subgroup of both $A$ and $B$ such that $H\le A\ast_C B$.
There is a graph of groups decomposition for $H$ such that each edge group embeds into $C$. 
Since $H$ is torsion-free and one-ended, this decomposition should essentially have only one vertex group.
\end{proof}

We denote by $P_n$ the path on $n$ vertices.
Let us recall from~\cite{css2008} the graphs $P_1(7)$ and $P_2(7)$, whose \emph{opposite graphs} are drawn in Figure~\ref{fig:forbidden}. 
The result in~\cite{css2008} implies that when $\Gam$ at most seven vertices, $A(\Gam)\in\mS$ if and only if $\Gam$ contains $C_6\opp$, $P_6\opp$, $P_1(7)$, $P_2(7)$ or $C_m$ for some $m\ge5$ as an induced subgraph.

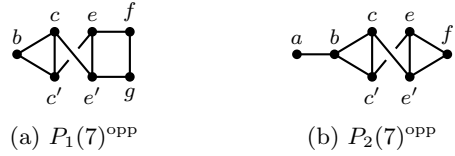
\begin{figure}[htb!]
  \tikzstyle {b}=[postaction=decorate,decoration={%
    markings,%
    mark=at position .57 with {\arrow{stealth};},%
    }]
  \tikzstyle {v}=[draw,shape=circle,fill=black,inner sep=0pt]
  \tikzstyle {bv}=[black,draw,shape=circle,fill=black,inner sep=1pt]
\tikzstyle{interrupt}=[
    postaction={
        decorate,
        decoration={markings,
                    mark= at position 0.5 
                          with
                          {
                            \fill[white] (-0.1,-0.1) rectangle (0.1,0.1);
                          }
                    }
                }
]  
  \tikzstyle{every edge}=[-,draw]
	\subfloat[(a) $P_1(7)\opp$]{
	\begin{tikzpicture}[thick]
    	  \draw  
	  (-.75,0) node [bv]  {} node[above] {$\scriptstyle b$}
   	  -- (-.25,.3) node [bv]  {} node[above] {$\scriptstyle c$}
	    (.25,.3) node  [bv]   {} node[above] {$\scriptstyle e$} 
    	  -- (.75,.3) node [bv]  {} node[above] {$\scriptstyle f$}
   	  -- (.75,-.3) node [bv]  {} node[below] {$\scriptstyle g$}
	  --  (.25,-.3) node  [bv]  {} node[below] {$\scriptstyle e'$} 
   	   (-.25,-.3) node [bv]  {} node[below] {$\scriptstyle c'$}
	  -- (-.75,0) node []  {}; 
    	  \draw[interrupt] 	(-.25,-.3) node [] {} -- (.25,.3) node []  {};		  
    	  \draw 	(-.25,.3) node [] {} --  (.25,-.3) node  []  {};	  	  
    	  \draw 	(-.25,-.3) node [] {} --  (-.25,.3) node  []  {};
    	  \draw 	(.25,-.3) node [] {} --  (.25,.3) node  []  {};	  
	\end{tikzpicture}	}
	\qquad
	\qquad
	\subfloat[(b) $P_2(7)\opp$]{
	\begin{tikzpicture}[thick]
    	  \draw  (-1.25,0) node  [bv]   {} node[above] {$\scriptstyle a$} 
    	  -- (-.75,0) node [bv]  {} node[above] {$\scriptstyle b$}
   	  -- (-.25,.3) node [bv]  {} node[above] {$\scriptstyle c$}
	   (.25,.3) node  [bv]   {} node[above] {$\scriptstyle e$} 
    	  -- (.75,0) node [bv]  {} node[above] {$\scriptstyle f$};
    	  \draw   (-.75,0) node []  {} 
   	  -- (-.25,-.3) node [bv]  {} node[below] {$\scriptstyle c'$}
	    (.25,-.3) node  [bv]  {} node[below] {$\scriptstyle e'$} 
    	  -- (.75,0) node []  {} node[above] {};
    	  \draw 	(-.25,-.3) node [] {} --  (-.25,.3) node  []  {};
    	  \draw 	(.25,-.3) node [] {} --  (.25,.3) node  []  {};	  
    	  \draw[interrupt] 	(-.25,-.3) node [] {} -- (.25,.3) node []  {};		  
    	  \draw 	(-.25,.3) node [] {} --  (.25,-.3) node  []  {};
	\end{tikzpicture}	}
\caption{Graphs from~\cite{css2008}.}
\label{fig:forbidden}
\end{figure}

We now consider some eight-vertex graphs.

\begin{lem}\label{lem:8v}
Let $\Phi_1,\Phi_2,\ldots,\Phi_5$ be the graphs whose opposite graphs are shown in Figure~\ref{fig:8v} (a) through (e).
Then $C(\Phi_i)\not\in\mS$ for each $i=1,2,\ldots,5$.
\end{lem}

\begin{proof}
We use the vertex labels shown in Figure~\ref{fig:8v}. We also set $H=\form{a,b,f,g}\le G=\form{a,b,c,e,f,g}\le C(P_7\opp)=\form{a,b,c,d,e,f,g}$ as considered in Figure~\ref{fig:7v} (c). In Example~\ref{exmp:p3}, we have seen that $(G,H)$ is small.

(\textbf{Case}  $\Phi_1$)
We can write  
$C(\Phi_1) = (\form{a,b,c}\times\form{e,f,g}\times\form{t})\ast_{\form{a,b}\times\form{f,g}}
(\form{a,b}\times\form{d}\times\form{f,g})$.
By Lemma~\ref{lem:small product} and Example~\ref{exmp:p3},
$(\form{a,b,c}\times\form{e,f,g}\times\form{t}, {\form{a,b}\times\form{f,g}})$ is small.
Since $\form{a,b}\times\form{d}\times\form{f,g}$ is virtually abelian,
it contains no $F_2$ and hence, $(\form{a,b}\times\form{d}\times\form{f,g},\form{a,b}\times\form{f,g})$ is small.
Lemma~\ref{lem:transverse} (1) implies that $C(\Phi_1)\not\in\mS$.

(\textbf{Case}  $\Phi_2$)
We note that $C(P_7\opp)\le C(\Phi_1)$ and so, $C(P_7\opp)\not\in\mS$.
Moreover, we have $C(\Phi_1) = C(P_7\opp)\ast_G /\fform{t^2}$ 
where $t$ is the stable generator.
By Lemma~\ref{lem:double}, $C(\Phi_2)=\form{a,b,c,d,e,f,g}\ast_{\form{a,b,c,e,f,g}}\form{a,b,c,d',e,f,g}\le C(\Phi_1)$.

(\textbf{Case}  $\Phi_3$)
Note that $C(P_7\opp) = G\ast_H/\fform{d^2}$ where $d$ is the stable generator. 
Hence, $C(\Phi_3) = \form{a,b,c,e,f,g}\ast_{\form{a,b,f,g}} \form{a,b,c',e',f,g}\le C(P_7\opp)$.

(\textbf{Case}  $\Phi_4$)
Let us consider $\psi\co  H\to H$ defined by $\psi(a)=a,\psi(b)=b,\psi(f)=g$ and $\psi(g)=f$. 
We see that $G\ast_{\psi\co H\to H} G
=\form{a,b,c,e,f,g}\ast_\psi \form{a,b,c',e',f'=g,g'=f} = C(\Phi_4)$.
Lemma~\ref{lem:transverse} (1) 
and Example~\ref{exmp:p3} imply that $C(\Phi_4)\not\in\mS$.

(\textbf{Case}  $\Phi_5$)
Similarly, define $\psi\co H\to H$ by $\psi(a)=b,\psi(b)=a,\psi(f)=g$ and $\psi(g)=f$. 
We again see that $G\ast_{\psi\co H\to H} G
=\form{a,b,c,e,f,g}\ast_\psi \form{a'=b,b'=a,c',e',f'=g,g'=f}= C(\Phi_5)\not\in\mS$.
\end{proof}

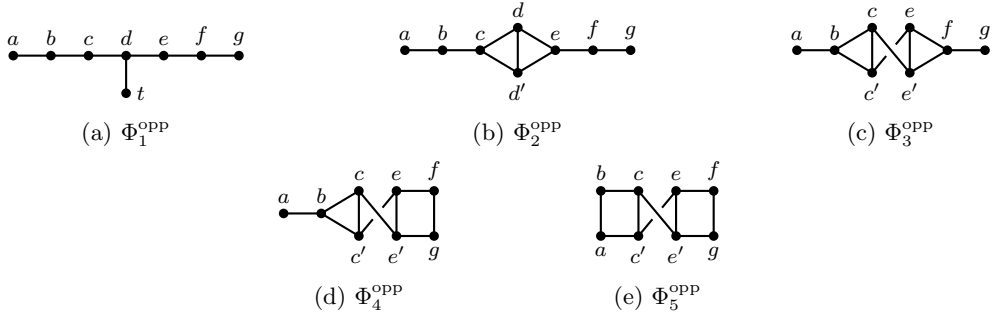
\begin{figure}[htb!]
  \tikzstyle {b}=[postaction=decorate,decoration={%
    markings,%
    mark=at position .57 with {\arrow{stealth};},%
    }]
  \tikzstyle {v}=[draw,shape=circle,fill=black,inner sep=0pt]
  \tikzstyle {bv}=[black,draw,shape=circle,fill=black,inner sep=1pt]
\tikzstyle{interrupt}=[
    postaction={
        decorate,
        decoration={markings,
                    mark= at position 0.5 
                          with
                          {
                            \fill[white] (-0.1,-0.1) rectangle (0.1,0.1);
                          }
                    }
                }
]  
  \tikzstyle{every edge}=[-,draw]
	\subfloat[(a) $\Phi_1\opp$]{
	\begin{tikzpicture}[thick]
    	  \draw  (-1.5,0) node  [bv]  (am) {} node[above] {$\scriptstyle a$} 
    	  -- (-1,0) node [bv] (ap) {} node[above] {$\scriptstyle b$}
   	  -- (-.5,0) node [bv] (ap) {} node[above] {$\scriptstyle c$}
   	  -- (0,0) node [bv] (ap) {} node[above] {$\scriptstyle d$}	  
   	  -- (.5,0) node [bv] (ap) {} node[above] {$\scriptstyle e$}	  
   	  -- (1,0) node [bv] (ap) {} node[above] {$\scriptstyle f$}
   	  -- (1.5,0) node [bv] (ap) {} node[above] {$\scriptstyle g$};
	  \draw (0,0) node [] {} -- (0,-.5) node [bv] {} node [right] {$\scriptstyle t$};
	\end{tikzpicture}
	}
	\qquad
	\qquad
	\subfloat[(b) $\Phi_2\opp$]{
	\begin{tikzpicture}[thick]
    	  \draw  (-1.5,0) node  [bv]   {} node[above] {$\scriptstyle a$} 
    	  -- (-1,0) node [bv]  {} node[above] {$\scriptstyle b$}
   	  -- (-.5,0) node [bv]  {} node[above] {$\scriptstyle c$}
   	  -- (0,.3) node [bv]  {} node[above] {$\scriptstyle d$}
	  --  (.5,0) node  [bv]   {} node[above] {$\scriptstyle e$} 
    	  -- (1,0) node [bv]  {} node[above] {$\scriptstyle f$}
   	  -- (1.5,0) node [bv]  {} node[above] {$\scriptstyle g$};	  
    	  \draw   (-.5,0) node []  {}
   	  -- (0,-.3) node [bv]  {} node[below] {$\scriptstyle d'$}
	  --  (.5,0) node  []  {};
    	  \draw 	(0,-.3) node [] {} --  (0,.3) node  []  {};
	\end{tikzpicture}	}
	\qquad
	\qquad
	\subfloat[(c) $\Phi_3\opp$]{
	\begin{tikzpicture}[thick]
    	  \draw  (-1.25,0) node  [bv]   {} node[above] {$\scriptstyle a$} 
    	  -- (-.75,0) node [bv]  {} node[above] {$\scriptstyle b$}
   	  -- (-.25,.3) node [bv]  {} node[above] {$\scriptstyle c$}
	   (.25,.3) node  [bv]   {} node[above] {$\scriptstyle e$} 
    	  -- (.75,0) node [bv]  {} node[above] {$\scriptstyle f$}
   	  -- (1.25,0) node [bv]  {} node[above] {$\scriptstyle g$};	  
    	  \draw   (-.75,0) node []  {} 
   	  -- (-.25,-.3) node [bv]  {} node[below] {$\scriptstyle c'$}
	    (.25,-.3) node  [bv]  {} node[below] {$\scriptstyle e'$} 
    	  -- (.75,0) node []  {} node[above] {};
    	  \draw 	(-.25,-.3) node [] {} --  (-.25,.3) node  []  {};
    	  \draw 	(.25,-.3) node [] {} --  (.25,.3) node  []  {};	  
    	  \draw[interrupt] 	(-.25,-.3) node [] {} -- (.25,.3) node []  {};		  
    	  \draw 	(-.25,.3) node [] {} --  (.25,-.3) node  []  {};
	\end{tikzpicture}	}
	\qquad
	\qquad
	\subfloat[(d) $\Phi_4\opp$]{
	\begin{tikzpicture}[thick]
    	  \draw  (-1.25,0) node  [bv]   {} node[above] {$\scriptstyle a$} 
    	  -- (-.75,0) node [bv]  {} node[above] {$\scriptstyle b$}
   	  -- (-.25,.3) node [bv]  {} node[above] {$\scriptstyle c$}
	    (.25,.3) node  [bv]   {} node[above] {$\scriptstyle e$} 
    	  -- (.75,.3) node [bv]  {} node[above] {$\scriptstyle f$}
   	  -- (.75,-.3) node [bv]  {} node[below] {$\scriptstyle g$}
	  --  (.25,-.3) node  [bv]  {} node[below] {$\scriptstyle e'$} 
   	   (-.25,-.3) node [bv]  {} node[below] {$\scriptstyle c'$}
	  -- (-.75,0) node []  {}; 
    	  \draw[interrupt] 	(-.25,-.3) node [] {} -- (.25,.3) node []  {};		  
    	  \draw 	(-.25,.3) node [] {} --  (.25,-.3) node  []  {};	  	  
    	  \draw 	(-.25,-.3) node [] {} --  (-.25,.3) node  []  {};
    	  \draw 	(.25,-.3) node [] {} --  (.25,.3) node  []  {};	  
	\end{tikzpicture}	}
	\qquad
	\qquad
	\subfloat[(e) $\Phi_5\opp$]{
	\begin{tikzpicture}[thick]
    	  \draw  (-.75,-.3) node  [bv]   {} node[below] {$\scriptstyle a$} 
    	  -- (-.75,.3) node [bv]  {} node[above] {$\scriptstyle b$}
   	  -- (-.25,.3) node [bv]  {} node[above] {$\scriptstyle c$}
	    (.25,.3) node  [bv]   {} node[above] {$\scriptstyle e$} 
    	  -- (.75,.3) node [bv]  {} node[above] {$\scriptstyle f$}
   	  -- (.75,-.3) node [bv]  {} node[below] {$\scriptstyle g$}
	  --  (.25,-.3) node  [bv]  {} node[below] {$\scriptstyle e'$} 
   	   (-.25,-.3) node [bv]  {} node[below] {$\scriptstyle c'$}
	  -- (-.75,-.3) node []  {}; 	  
    	  \draw 	(-.25,-.3) node [] {} --  (-.25,.3) node  []  {};
    	  \draw 	(.25,-.3) node [] {} --  (.25,.3) node  []  {};	
    	  \draw[interrupt] 	(-.25,-.3) node [] {} -- (.25,.3) node []  {};		  
    	  \draw 	(-.25,.3) node [] {} --  (.25,-.3) node  []  {};	  	  
	\end{tikzpicture}	}
\caption{Some eight-vertex graphs in Lemma~\ref{lem:8v}.}
\label{fig:8v}
\end{figure}

\begin{proof}[Proof of Theorem~\ref{thm:7v}]
We use notations from Lemma~\ref{lem:8v} and Figure~\ref{fig:8v}.
The backward direction is a restatement of Corollaries~\ref{cor:cycle} and~\ref{cor:anticycle}.

For the forward direction, suppose $\Gam$ is weakly chordal and $C(\Gam)\in\mS$.
Since $[C(\Gam),C(\Gam)]\le [A(\Gam),A(\Gam)]$,
we see that $A(\Gam)\in\mS$.
By the result in~\cite[Section 7]{css2008}, $\Gam$ contains $P_6\opp, P_1(7)$ or $P_2(7)$ as an induced subgraph; see Figure~\ref{fig:forbidden}.
Since $P_1(7)\le \Phi_4$ and $P_2(7)\le \Phi_3$, Lemma~\ref{lem:8v} implies that $C(P_1(7))\not\in\mS$ and $C(P_2(7))\not\in\mS$.
Hence $P_6\opp\le\Gam$. Moreover, $|V(\Gam)|=7$ since $C(P_6\opp)\le C(P_7\opp)\le C(\Phi_1)\not\in\mS$.

We fix the vertex labels of $P_6\le\Gam\opp$ as in Figure~\ref{fig:7v} (b) and let $V(\Gam)\ssm V(P_6\opp)=\{t\}$.
By $\val{t}$, we will mean the valence of $t$ in $\Gam$.

(\textbf{Case} $\val{t}=0$ or $1$)
Note $C(\Gam)=C(P_6\opp)\ast\Z_2$ or $C(\Gam)=C(P_6\opp)\ast_{\mathbb{Z}_2} (\mathbb{Z}_2)^2$.
By applying Lemma~\ref{lem:finite amalgam}, we obtain a contradiction that $C(\Gam)\not\in\mS$.
 
(\textbf{Case}  $\val{t}=2$)
In $\Gam$, the vertex $t$ is joined to two vertices, say $x,y$ of $P_6\opp$. 
We may write $C(\Gam)=C(P_6\opp)\ast_{\form{x,y}} \form{x,y,t}$.
If $x$ and $y$ are adjacent in $P_6\opp$, then 
 $C(\Gam)=C(P_6\opp)\ast_{(\mathbb{Z}_2)^2} (\mathbb{Z}_2)^3$; then, Lemma~\ref{lem:finite amalgam} implies that  $C(\Gam)\not\in\mS$.
So, $x$ and $y$ are adjacent in $P_6\le\Gam\opp$. 
Since $\Gam\opp$ has no induced $C_5$, we should have $\Gam\cong\Lambda_1$; see Figure~\ref{fig:7v full} (a). 
We have $C(\Lambda_1) = (\form{a,b}\times \form{d,e,f,t})\ast_{\form{a,e,f}}(\form{a}\times\form{c}\times\form{e,f})$.
By Lemmas~\ref{lem:links} and~\ref{lem:small product}, 
$(\form{a,b}\times\form{d,e,f,t},\form{a}\times\form{e,f})$ is small. Since $\form{a,c,e,f}$ is virtually abelian, we have $C(\Lambda_1)\not\in\mS$ by~Lemma~\ref{lem:transverse} (1). This is a contradiction.

(\textbf{Case}  $\Gam=\Lambda_2$; see Figure~\ref{fig:7v full} (b))
By Lemma~\ref{lem:double},
$C(\Lambda_2) = \form{a,b,c,d,e,f}\ast_{\form{b,c,d,e,f}}\form{t,b,c,d,e,f}$ embeds into
$\form{a,b,c,d,e,f}\ast_{\form{b,c,d,e,f}}/\fform{s^2}\cong C(P_7\opp)\not\in\mS$, where $s$ denotes the stable generator.

(\textbf{Case}  $\Gam=\Lambda_3$; see Figure~\ref{fig:7v full} (c))
Using the vertex labels of $P_6\opp$ in Figure~\ref{fig:7v} (b), 
we see that $C(\Lambda_3) \cong GP(P_6\opp,\{G_a=\Z_2\times\Z_2,G_b=G_c=\cdots=G_f=\Z_2\})$. Corollary~\ref{cor:inclusion} implies that $C(\Lambda_3)$ virtually embeds into $GP(P_6\opp,\{G_a=\Z_2\ast\Z_2,G_b=G_c=\cdots=G_f=\Z_2\})\cong C(\Lambda_2)\not\in\mS$.

(\textbf{Case}  $\val{t}=3$)
If $a$ and $f$ are both adjacent to $t$ in $\Gam\opp$, then only one of $b,c,d,e$ are adjacent to $t$ in $\Gam\opp$. This implies that $\Gam\opp$ contains an induced $C_5$, and hence a contradiction. So, we may assume $a$ is not adjacent to $t$ in $\Gam\opp$.
Let us say $x$ and $y$ are the other two vertices of $P_6\le\Gam\opp$ that are non-adjacent to $t$ in $\Gam\opp$.
If $\{a,x,y\}$ are pairwise non-adjacent in $P_6$, then 
$C(\Gam) = C(P_6\opp)\ast_{\form{x,y,z}} \form{t,x,y,z}=C(P_6\opp)\ast_{(\mathbb{Z}_2)^3} (\mathbb{Z}_2)^4\not\in\mS$. So there exist at least two vertices in $\{a,x,y\}$ that are adjacent in $P_6$. 
If $b\in\{x,y\}$, then $\Gam \cong \Lambda_i$ for $i=4,5,6,7$; see Figure~\ref{fig:7v full} (d) through (g).
If $b\not\in\{x,y\}$, then $x$ and $y$ must be adjacent in $P_6$.
Since $C_5\not\le\Gam\opp$, we would have a graph isomorphism $\Gam\cong \Lambda_5$. 

If $\Gam\cong \Lambda_4$, then $C(\Gam) = \form{a,b,c,d,e,f}\ast_{\form{a,b,c,d,f}}\form{a,b,c,d,t,f}
\le \form{a,b,c,d,e,f}\ast_{\form{a,b,c,d,f}}/\fform{s^2}\cong C(\Lambda_3)$ where $s$ is the stable generator.
Similarly if $\Gam=\Lambda_5$, then $C(\Gam)= \form{a,b,c,d,e,f}\ast_{\form{a,b,c,e,f}}\form{a,b,c,t,e,f}\le\form{a,b,c,d,e,f}\ast_{\form{a,b,c,e,f}}/\fform{s^2}\cong C(\Lambda)$ where $s$ is the stable generator and $\Lambda\opp$ is the subgraph of $\Phi_1\opp$ induced by $\{a,b,c,d,e,f,t\}$; see~Figure~\ref{fig:8v} (a).

Suppose $\Gam=\Lambda_6$. 
Then $C(\Gam)= (\form{a,b,c}\times\form{e,f})\ast_{\form{a,b}\times\form{f}}
(\form{a,b}\times\form{d}\times\form{t,f})$. 
Since $\form{a,b}\times\form{d}\times\form{t,f}$ is virtually abelian, 
Example~\ref{exmp:p3} implies that $C(\Gam)\not\in\mS$.

Consider the case $\Gam=\Lambda_7$.
Then $\Gam\opp$ is obtained from $\Phi_4\opp$ in Figure~\ref{fig:8v} (d) by contracting $\{c,c'\}$ to a vertex. By Theorem~\ref{thm:cocont}, $C(\Gam)$ embeds into $C(\Phi_4)$.

(\textbf{Case}  $\val{t}=4$)
Let $x$ and $y$ be the two vertices adjacent to $t$ in $\Gam\opp$.
Since $C_5\not\le\Gam\opp$, we see that $d(x,y)\le 2$ in $P_6$. So $\Gam\cong\Lambda_i$ for $i=2,8,9,10,11$; see Figure~\ref{fig:7v full}.

By Corollary~\ref{cor:inclusion}, $C(\Lambda_ 8)$ and $C(\Lambda_9)$ virtually embed into $C(\Lambda_4)$ and $C(\Lambda_5)$, respectively.
If $\Gam=\Lambda_{10}$, write  
$C(\Lambda_{10}) = (\form{a,b,t}\times\form{d,e,f})\ast_{\form{a}\times\form{e,f}}
(\form{a}\times\form{c}\times\form{e,f})\not\in\mS$.
We see that $\Phi_2\opp$ contracts onto $\Lambda_{11}\opp$; see~Figure~\ref{fig:8v} (b) and Figure~\ref{fig:7v full} (k). So, $C(\Lambda_{11})\le C(\Phi_2)$.

(\textbf{Case}  $\val{t}=5$)
Either $\Gam\le \Phi_1$ or $\Gam\cong\Lambda_3$.

(\textbf{Case}  $\val{t}=6$)
Note $C(\Gam)=C(P_6\opp)\times\form{t}$.

\end{proof}

\begin{figure}[htb!]
  \tikzstyle {b}=[postaction=decorate,decoration={%
    markings,%
    mark=at position .57 with {\arrow{stealth};},%
    }]
  \tikzstyle {v}=[draw,shape=circle,fill=black,inner sep=0pt]
  \tikzstyle {bv}=[black,draw,shape=circle,fill=black,inner sep=1pt]
\tikzstyle{interrupt}=[
    postaction={
        decorate,
        decoration={markings,
                    mark= at position 0.5 
                          with
                          {
                            \fill[white] (-0.1,-0.1) rectangle (0.1,0.1);
                          }
                    }
                }
]  
  \tikzstyle{every edge}=[-,draw]
	\subfloat[(a) $\Lambda_1\opp$]{
	\begin{tikzpicture}[thick]
    	  \draw  (-1.5,0) node  [bv]  (am) {} node[above] {$\scriptstyle a$} 
    	  -- (-1,0) node [bv] (ap) {} node[above] {$\scriptstyle b$}
   	  -- (-.5,0) node [bv] (ap) {} node[above] {$\scriptstyle c$}
   	  -- (0,0) node [bv] (ap) {} node[above] {$\scriptstyle d$}	  
   	  -- (.5,0) node [bv] (ap) {} node[above] {$\scriptstyle e$}	  
   	  -- (1,0) node [bv] (ap) {} node[above] {$\scriptstyle f$};
	  \draw (-.5,0) node [] {} -- (.25,-.5) node [bv] {} node [right] {$\scriptstyle t$};
	  \draw (0,0) node [] {} -- (.25,-.5) node [] {};
	  \draw (.5,0) node [] {} -- (.25,-.5) node [] {};	  
	  \draw (1,0) node [] {} -- (.25,-.5) node [] {};	  	  	  	  
	\end{tikzpicture}
	}
	\qquad
	\subfloat[(b) $\Lambda_2\opp$]{
	\begin{tikzpicture}[thick]
    	  \draw  (-1.5,0) node  [bv]  (am) {} node[above] {$\scriptstyle a$} 
    	  -- (-1,0) node [bv] (ap) {} node[above] {$\scriptstyle b$}
   	  -- (-.5,0) node [bv] (ap) {} node[above] {$\scriptstyle c$}
   	  -- (0,0) node [bv] (ap) {} node[above] {$\scriptstyle d$}	  
   	  -- (.5,0) node [bv] (ap) {} node[above] {$\scriptstyle e$}	  
   	  -- (1,0) node [bv] (ap) {} node[above] {$\scriptstyle f$};
	  \draw (-1.5,0) node [] {} -- (-1,-.5) node [bv] {} node [right] {$\scriptstyle t$};
	  \draw (-1,0) node [] {} -- (-1,-.5) node [] {};	  
	\end{tikzpicture}
	}		
	\qquad		
	\subfloat[(c) $\Lambda_3\opp$]{
	\begin{tikzpicture}[thick]
    	  \draw  (-1.5,0) node  [bv]  (am) {} node[above] {$\scriptstyle a$} 
    	  -- (-1,0) node [bv] (ap) {} node[above] {$\scriptstyle b$}
   	  -- (-.5,0) node [bv] (ap) {} node[above] {$\scriptstyle c$}
   	  -- (0,0) node [bv] (ap) {} node[above] {$\scriptstyle d$}	  
   	  -- (.5,0) node [bv] (ap) {} node[above] {$\scriptstyle e$}	  
   	  -- (1,0) node [bv] (ap) {} node[above] {$\scriptstyle f$};
	  \draw (-1,0) node [] {} -- (-1,-.5) node [bv] {} node [right] {$\scriptstyle t$};
	\end{tikzpicture}
	}		
	\qquad
	\subfloat[(d) $\Lambda_4\opp$]{
	\begin{tikzpicture}[thick]
    	  \draw  (-1.5,0) node  [bv]  (am) {} node[above] {$\scriptstyle a$} 
    	  -- (-1,0) node [bv] (ap) {} node[above] {$\scriptstyle b$}
   	  -- (-.5,0) node [bv] (ap) {} node[above] {$\scriptstyle c$}
   	  -- (0,0) node [bv] (ap) {} node[above] {$\scriptstyle d$}	  
   	  -- (.5,0) node [bv] (ap) {} node[above] {$\scriptstyle e$}	  
   	  -- (1,0) node [bv] (ap) {} node[above] {$\scriptstyle f$};
	  \draw (0,0) node [] {} -- (.25,-.5) node [bv] {} (.3,-.5) node [right] {$\scriptstyle t$};
	  \draw (.5,0) node [] {} -- (.25,-.5) node [] {};	  
	  \draw (1,0) node [] {} -- (.25,-.5) node [] {};	  	  	  	  
	\end{tikzpicture}
	}
	\qquad
	\subfloat[(e) $\Lambda_5\opp$]{
	\begin{tikzpicture}[thick]
    	  \draw  (-1.5,0) node  [bv]  (am) {} node[above] {$\scriptstyle a$} 
    	  -- (-1,0) node [bv] (ap) {} node[above] {$\scriptstyle b$}
   	  -- (-.5,0) node [bv] (ap) {} node[above] {$\scriptstyle c$}
   	  -- (0,0) node [bv] (ap) {} node[above] {$\scriptstyle d$}	  
   	  -- (.5,0) node [bv] (ap) {} node[above] {$\scriptstyle e$}	  
   	  -- (1,0) node [bv] (ap) {} node[above] {$\scriptstyle f$};
	  \draw (0,0) node [] {} -- (.25,-.5) node [bv] {} node [right] {$\scriptstyle t$};
	  \draw (.5,0) node [] {} -- (.25,-.5) node [] {};	  
	  \draw (-.5,0) node [] {} -- (.25,-.5) node [] {};	  	  	  	  
	\end{tikzpicture}
	}
	\qquad
	\subfloat[(f) $\Lambda_6\opp$]{
	\begin{tikzpicture}[thick]
    	  \draw  (-1.5,0) node  [bv]  (am) {} node[above] {$\scriptstyle a$} 
    	  -- (-1,0) node [bv] (ap) {} node[above] {$\scriptstyle b$}
   	  -- (-.5,0) node [bv] (ap) {} node[above] {$\scriptstyle c$}
   	  -- (0,0) node [bv] (ap) {} node[above] {$\scriptstyle d$}	  
   	  -- (.5,0) node [bv] (ap) {} node[above] {$\scriptstyle e$}	  
   	  -- (1,0) node [bv] (ap) {} node[above] {$\scriptstyle f$};
	  \draw (-.5,0) node [] {} -- (.25,-.5) node [bv] {} (.3,-.5) node [right] {$\scriptstyle t$};
	  \draw (.5,0) node [] {} -- (.25,-.5) node [] {};	  
	  \draw (1,0) node [] {} -- (.25,-.5) node [] {};	  	  	  	  
	\end{tikzpicture}
	}
	\qquad
	\subfloat[(g) $\Lambda_7\opp$]{
	\begin{tikzpicture}[thick]
    	  \draw  (-1.5,0) node  [bv]  (am) {} node[above] {$\scriptstyle a$} 
    	  -- (-1,0) node [bv] (ap) {} node[above] {$\scriptstyle b$}
   	  -- (-.5,0) node [bv] (ap) {} node[above] {$\scriptstyle c$}
   	  -- (0,0) node [bv] (ap) {} node[above] {$\scriptstyle d$}	  
   	  -- (.5,0) node [bv] (ap) {} node[above] {$\scriptstyle e$}	  
   	  -- (1,0) node [bv] (ap) {} node[above] {$\scriptstyle f$};
	  \draw (0,0) node [] {} -- (.25,-.5) node [bv] {} (.33,-.5) node [right] {$\scriptstyle t$};
	  \draw (-.5,0) node [] {} -- (.25,-.5) node [] {};	  
	  \draw (1,0) node [] {} -- (.25,-.5) node [] {};	  	  	  	  
	\end{tikzpicture}
	}	
	\qquad
	\subfloat[(h) $\Lambda_8\opp$]{
	\begin{tikzpicture}[thick]
    	  \draw  (-1.5,0) node  [bv]  (am) {} node[above] {$\scriptstyle a$} 
    	  -- (-1,0) node [bv] (ap) {} node[above] {$\scriptstyle b$}
   	  -- (-.5,0) node [bv] (ap) {} node[above] {$\scriptstyle c$}
   	  -- (0,0) node [bv] (ap) {} node[above] {$\scriptstyle d$}	  
   	  -- (.5,0) node [bv] (ap) {} node[above] {$\scriptstyle e$}	  
   	  -- (1,0) node [bv] (ap) {} node[above] {$\scriptstyle f$};
	  \draw (-1.5,0) node [] {} -- (-1,-.5) node [bv] {} node [right] {$\scriptstyle t$};
	  \draw (-.5,0) node [] {} -- (-1,-.5) node [] {};	  
	\end{tikzpicture}
	}	
	\qquad
	\subfloat[(i) $\Lambda_9\opp$]{
	\begin{tikzpicture}[thick]
    	  \draw  (-1.5,0) node  [bv]  (am) {} node[above] {$\scriptstyle a$} 
    	  -- (-1,0) node [bv] (ap) {} node[above] {$\scriptstyle b$}
   	  -- (-.5,0) node [bv] (ap) {} node[above] {$\scriptstyle c$}
   	  -- (0,0) node [bv] (ap) {} node[above] {$\scriptstyle d$}	  
   	  -- (.5,0) node [bv] (ap) {} node[above] {$\scriptstyle e$}	  
   	  -- (1,0) node [bv] (ap) {} node[above] {$\scriptstyle f$};
	  \draw (-1,0) node [] {} -- (-.5,-.5) node [bv] {} node [right] {$\scriptstyle t$};
	  \draw (-0,0) node [] {} -- (-.5,-.5) node [] {};	  
	\end{tikzpicture}
	}	
	\qquad
	\subfloat[(j) $\Lambda_{10}\opp$]{
	\begin{tikzpicture}[thick]
    	  \draw  (-1.5,0) node  [bv]  (am) {} node[above] {$\scriptstyle a$} 
    	  -- (-1,0) node [bv] (ap) {} node[above] {$\scriptstyle b$}
   	  -- (-.5,0) node [bv] (ap) {} node[above] {$\scriptstyle c$}
   	  -- (0,0) node [bv] (ap) {} node[above] {$\scriptstyle d$}	  
   	  -- (.5,0) node [bv] (ap) {} node[above] {$\scriptstyle e$}	  
   	  -- (1,0) node [bv] (ap) {} node[above] {$\scriptstyle f$};
	  \draw (-1,0) node [] {} -- (-1,-.5) node [bv] {} node [right] {$\scriptstyle t$};
	  \draw (-.5,0) node [] {} -- (-1,-.5) node [] {};	  
	\end{tikzpicture}
	}		
	\qquad
	\subfloat[(k) $\Lambda_{11}\opp$]{
	\begin{tikzpicture}[thick]
    	  \draw  (-1.5,0) node  [bv]  (am) {} node[above] {$\scriptstyle a$} 
    	  -- (-1,0) node [bv] (ap) {} node[above] {$\scriptstyle b$}
   	  -- (-.5,0) node [bv] (ap) {} node[above] {$\scriptstyle c$}
   	  -- (0,0) node [bv] (ap) {} node[above] {$\scriptstyle d$}	  
   	  -- (.5,0) node [bv] (ap) {} node[above] {$\scriptstyle e$}	  
   	  -- (1,0) node [bv] (ap) {} node[above] {$\scriptstyle f$};
	  \draw (-.5,0) node [] {} -- (-.5,-.5) node [bv] {} node [right] {$\scriptstyle t$};
	  \draw (-0,0) node [] {} -- (-.5,-.5) node [] {};	  
	\end{tikzpicture}
	}						
\caption{Seven-vertex graphs in Theorem~\ref{thm:7v}.}
\label{fig:7v full}
\end{figure}
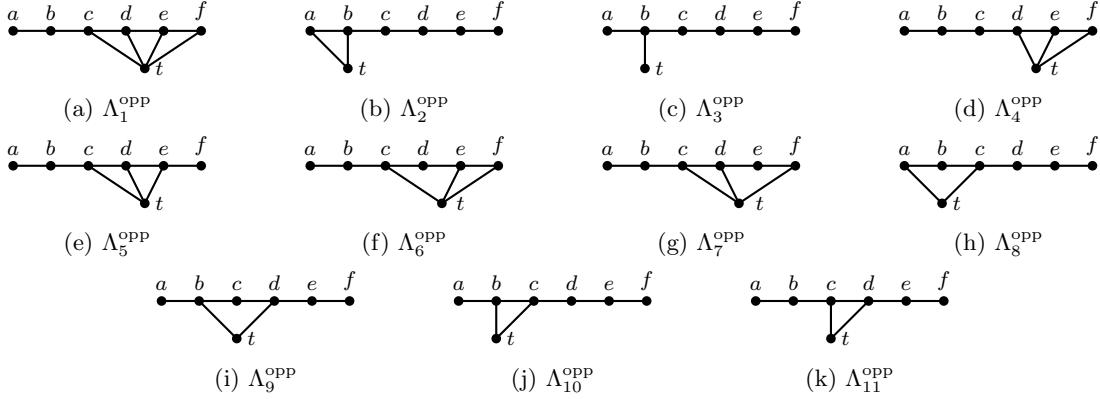

\begin{rem}\label{rem:w chordal}
\be
\item
Theorem~\ref{thm:7v} is not true if $\Gam$ has more than seven vertices.
That is, there exists a weakly chordal graph $\Gam$ such that $C(\Gam)\in\mS$.
For example, let $\Gam$ be the graph whose opposite graph is shown in Figure~\ref{fig:p6double}.
By~\cite{DJ2000}, $A(P_6\opp)$ is an index-64 subgroup of $C(\Gam)$ and so, $C(\Gam)\in\mS$.
\begin{figure}[htb!]
  \tikzstyle {b}=[postaction=decorate,decoration={%
    markings,%
    mark=at position .57 with {\arrow{stealth};},%
    }]
  \tikzstyle {v}=[draw,shape=circle,fill=black,inner sep=0pt]
  \tikzstyle {bv}=[black,draw,shape=circle,fill=black,inner sep=1pt]
\tikzstyle{interrupt}=[
    postaction={
        decorate,
        decoration={markings,
                    mark= at position 0.5 
                          with
                          {
                            \fill[white] (-0.1,-0.1) rectangle (0.1,0.1);
                          }
                    }
                }
]  
  \tikzstyle{every edge}=[-,draw]
	\begin{tikzpicture}[thick]
    	  \draw  (-1.5,0) node  [bv]  (am) {}
    	  -- (-1,0) node [bv] (ap) {}
   	  -- (-.5,0) node [bv] (ap) {}
   	  -- (0,0) node [bv] (ap) {}
   	  -- (.5,0) node [bv] (ap) {}
   	  -- (1,0) node [bv] (ap) {};
	  \draw (-1.5,0) node [] {} -- (-1.5,-.5) node [bv] {};
	  \draw (-1,0) node [] {} -- (-1,-.5) node [bv] {};
	  \draw (-.5,0) node [] {} -- (-.5,-.5) node [bv] {};
	  \draw (0,0) node [] {} -- (0,-.5) node [bv] {};
	  \draw (.5,0) node [] {} -- (.5,-.5) node [bv] {};
	  \draw (1,0) node [] {} -- (1,-.5) node [bv] {};
	\end{tikzpicture}					
\caption{The graph $\Gam\opp$ in Proposition~\ref{rem:w chordal}.}
\label{fig:p6double}
\end{figure}
\item
There exists a graph $\Gam$ such that $A(\Gam)\in\mS$ and $C(\Gam)\not\in\mS$.
For instance, we may set $\Gam$ as one of the graphs $P_6\opp, P_1(7)$ or $P_2(7)$.
\ee
\end{rem}

\begin{prob}
\be
\item
Does there exist a graph $\Gam$ such that $[A(\Gam),A(\Gam)]\not\in\mS$ while $A(\Gam)\in\mS$?
\item
Does there exist a graph $\Gam$ such that $[A(\Gam),A(\Gam)]\in\mS$ while $C(\Gam)\not\in\mS$?
\ee
\end{prob}

\section{Closure under graph products}\label{sec:gp closure}
A group $G$ is \emph{periodic} if every element of $G$ has a torsion. The following is well-known.

\begin{lem}[{\cite{GH1990}}]\label{lem:hyp}
A word-hyperbolic group does not have an infinite periodic subgroup.
\end{lem}

Let us denote by $\mX$ the class of finitely generated groups that are either
\begin{enumerate}[(i)]
\item not one-ended, or
\item not word-hyperbolic, or
\item containing hyperbolic surface groups.
\end{enumerate}

An affirmative answer to Question~\ref{que:gromov} is equivalent to saying that every finitely generated group is in $\mX$. How large do we know $\mX$ is? We prove that $\mX$ is closed under graph products.

\begin{thm}\label{thm:closure}
If $\mG=\{G_v\co v\in V(\Gam)\}$ is a collection of groups in $\mX$, then $GP(\Gam,\mG)$ is in $\mX$.
\end{thm}

\begin{proof}
Suppose that $G=GP(\Gam,\mG)$ is one-ended and word-hyperbolic.
We may assume that each $G_v$ is nontrivial.
If $\Gam$ contains an induced cycle of length at least five, Corollary~\ref{cor:cycle} implies that $G\in\mS$.
If $\Gam$ contains an induced square, whose vertices are denoted as $a,b,c$ and $d$ cyclically, then $G$ would contain
$(G_a\ast G_c)\times (G_b\ast G_d)\ge \Z\times\Z$. 
So from now on, we will assume that $C_n\not\le\Gam$ for every $n\ge4$; 
namely, $\Gam$ is a \emph{chordal} graph~\cite{golumbic2004}.

Suppose that $\Gam$ is complete.
Then $G$ is the direct product of its vertex groups.
Since $G$ is one-ended,
at least one vertex group, say $G_a$, must be infinite.
By Lemma~\ref{lem:hyp},
each infinite vertex group of $G$ contains $\Z$. As 
$G$ does not contain $\Z\times \Z$, exactly one vertex group is infinite. 
Then $G$ is virtually $G_a$, and so, $G_a$ is one-ended hyperbolic. Since $G_a\in\mX$, we have 
$G_a\in \mS$.

Now, assume that $\Gam$ is not complete. Since $\Gam$ is chordal,
 $\Gam$ can be written as $\Gam = \Gam_1\cup \Gam_2$ 
for some induced subgraphs $\Gam_1,\Gam_2$ such that  $\Gam_0=\Gam_1\cap\Gam_2$ is 
complete~\cite{dirac1961}. 
We choose a minimal such $\Gam_0$.
If all the vertex groups of $\Gam_0$ are finite, then $G$ splits over a finite group,
and hence $G$ has more than one ends. 
So $G_a$ is infinite for some $a\in V(\Gam_0)$.
By minimality of $\Gam_0$, we can find $a_i\in\Gam_i\ssm \Gam_0$
such that $a_i$ is adjacent to $a$ for $i=1,2$. This implies that
$G$ contains $G_a\times(G_{a_1}\ast G_{a_2})$, and hence, $\Z\times \Z$.
This is a contradiction.
\end{proof}

\begin{rem}
\be
\item
Several other classes of groups are known to be closed under the graph product operation.
These classes include residually finite groups~\cite{green1990}, semihyperbolic groups~\cite{ab1995}, automatic groups~\cite{hm1995} and diagram groups~\cite{gs2006}. 
Meier characterized exactly when a graph product of word-hyperbolic groups is word-hyperbolic~\cite{meier1996}.
\item
Every $3$--manifold group is in $\mX$. 
To see this, suppose $M$ is a $3$--manifold such that $\pi_1(M)$ is one-ended and word-hyperbolic.
We may assume $M$ is orientable by taking a double cover if necessary.
By the Loop Theorem, either $M$ has a hyperbolic incompressible boundary component or $M$ is closed possibly after capping off spherical boundary components.  
If $M$ is closed, Perelman's geometrization theorem implies that $M$ is a closed hyperbolic $3$--manifold; then, the work of Kahn and Markovic~\cite{KM2009} implies that $\pi_1(M)\in\mS$.
\ee
\end{rem}

\def\cprime{$'$}
\providecommand{\bysame}{\leavevmode\hbox to3em{\hrulefill}\thinspace}
\providecommand{\MR}{\relax\ifhmode\unskip\space\fi MR }
\providecommand{\MRhref}[2]{%
  \href{http://www.ams.org/mathscinet-getitem?mr=#1}{#2}
}
\providecommand{\href}[2]{#2}


\begin{thebibliography}{10}

\bibitem{Agol2008}
Ian Agol, \emph{Criteria for virtual fibering}, J. Topol. \textbf{1} (2008),
  no.~2, 269--284. \MR{2399130 (2009b:57033)}

\bibitem{ab1995}
Juan~M. Alonso and Martin~R. Bridson, \emph{Semihyperbolic groups}, Proc.
  London Math. Soc. (3) \textbf{70} (1995), no.~1, 56--114. \MR{1300841
  (95j:20033)}

\bibitem{bell2012}
Robert~W Bell, \emph{Combinatorial methods for detecting surface subgroups in
  right-angled {Artin} groups}, ISRN Algebra (2011), Article ID 102029, 6
  pages.

\bibitem{bh1999}
Martin~R. Bridson and Andr{\'e} Haefliger, \emph{Metric spaces of non-positive
  curvature}, Grundlehren der Mathematischen Wissenschaften [Fundamental
  Principles of Mathematical Sciences], vol. 319, Springer-Verlag, Berlin,
  1999. \MR{1744486 (2000k:53038)}

\bibitem{calegari2008}
Danny Calegari, \emph{Surface subgroups from homology}, Geom. Topol.
  \textbf{12} (2008), no.~4, 1995--2007. \MR{2431013 (2009d:20104)}

\bibitem{charney2000}
Ruth Charney, \emph{The {T}its conjecture for locally reducible {A}rtin
  groups}, Internat. J. Algebra Comput. \textbf{10} (2000), no.~6, 783--797.
  \MR{1809385 (2002d:20057)}

\bibitem{charney2007}
\bysame, \emph{An introduction to right-angled {A}rtin groups}, Geom. Dedicata
  \textbf{125} (2007), 141--158. \MR{2322545 (2008f:20076)}

\bibitem{CRSV2010}
Ruth Charney, Kim Ruane, Nathaniel Stambaugh, and Anna Vijayan, \emph{The
  automorphism group of a graph product with no {SIL}}, Illinois J. Math.
  \textbf{54} (2010), no.~1, 249--262. \MR{2776995}

\bibitem{css2008}
John Crisp, Michah Sageev, and Mark Sapir, \emph{Surface subgroups of
  right-angled {A}rtin groups}, Internat. J. Algebra Comput. \textbf{18}
  (2008), no.~3, 443--491. \MR{2422070 (2009f:20054)}

\bibitem{CW2004}
John Crisp and Bert Wiest, \emph{Embeddings of graph braid and surface groups
  in right-angled {A}rtin groups and braid groups}, Algebr. Geom. Topol.
  \textbf{4} (2004), 439--472. \MR{2077673 (2005e:20052)}

\bibitem{DJ2000}
Michael~W. Davis and Tadeusz Januszkiewicz, \emph{Right-angled {A}rtin groups
  are commensurable with right-angled {C}oxeter groups}, J. Pure Appl. Algebra
  \textbf{153} (2000), no.~3, 229--235. \MR{1783167 (2001m:20056)}

\bibitem{diestel2010}
Reinhard Diestel, \emph{Graph theory}, fourth ed., Graduate Texts in
  Mathematics, vol. 173, Springer, Heidelberg, 2010. \MR{2744811 (2011m:05002)}

\bibitem{dirac1961}
G.~A. Dirac, \emph{On rigid circuit graphs}, Abh. Math. Sem. Univ. Hamburg
  \textbf{25} (1961), 71--76. \MR{0130190 (24 \#A57)}

\bibitem{droms2003}
Carl Droms, \emph{A complex for right-angled {C}oxeter groups}, Proc. Amer.
  Math. Soc. \textbf{131} (2003), no.~8, 2305--2311 (electronic). \MR{1974626
  (2004e:20066)}

\bibitem{FT2012}
D~Futer and A~Thomas, \emph{Surface quotients of hyperbolic buildings}, Int.
  Math. Res. Not. IMRN, Advance Access published March 16, 2011,
  doi:10.1093/imrn/rnr028.

\bibitem{GH1990}
{\'E}.~Ghys and P.~de~la Harpe (eds.), \emph{Sur les groupes hyperboliques
  d'apr\`es {M}ikhael {G}romov}, Progress in Mathematics, vol.~83, Birkh\"auser
  Boston Inc., Boston, MA, 1990, Papers from the Swiss Seminar on Hyperbolic
  Groups held in Bern, 1988. \MR{1086648 (92f:53050)}

\bibitem{golumbic2004}
Martin~Charles Golumbic, \emph{Algorithmic graph theory and perfect graphs},
  second ed., Annals of Discrete Mathematics, vol.~57, Elsevier Science B.V.,
  Amsterdam, 2004, With a foreword by Claude Berge. \MR{2063679 (2005e:05061)}

\bibitem{GLR2004}
C.~McA. Gordon, D.~D. Long, and A.~W. Reid, \emph{Surface subgroups of
  {C}oxeter and {A}rtin groups}, J. Pure Appl. Algebra \textbf{189} (2004),
  no.~1-3, 135--148. \MR{2038569 (2004k:20077)}

\bibitem{GW2010}
Cameron Gordon and Henry Wilton, \emph{On surface subgroups of doubles of free
  groups}, J. Lond. Math. Soc. (2) \textbf{82} (2010), no.~1, 17--31.
  \MR{2669638 (2011k:20085)}

\bibitem{green1990}
E.~Green, \emph{Graph products of groups},  (1990), PhD thesis, University of
  Leeds.

\bibitem{gromov1987}
M.~Gromov, \emph{Hyperbolic groups}, Essays in group theory, Math. Sci. Res.
  Inst. Publ., vol.~8, Springer, New York, 1987, pp.~75--263. \MR{919829
  (89e:20070)}

\bibitem{gs2006}
V.~S. Guba and M.~V. Sapir, \emph{Diagram groups are totally orderable}, J.
  Pure Appl. Algebra \textbf{205} (2006), no.~1, 48--73. \MR{2193191
  (2006j:20057)}

\bibitem{haglund2008}
Fr{\'e}d{\'e}ric Haglund, \emph{Finite index subgroups of graph products},
  Geom. Dedicata \textbf{135} (2008), 167--209. \MR{2413337 (2009d:20098)}

\bibitem{HW2008}
Fr{\'e}d{\'e}ric Haglund and Daniel~T. Wise, \emph{Special cube complexes},
  Geom. Funct. Anal. \textbf{17} (2008), no.~5, 1551--1620. \MR{2377497
  (2009a:20061)}

\bibitem{HW2011}
\bysame, \emph{Coxeter groups are virtually special}, Adv. Math. \textbf{224}
  (2010), no.~5, 1890--1903. \MR{2646113 (2011g:20060)}

\bibitem{hm1995}
Susan Hermiller and John Meier, \emph{Algorithms and geometry for graph
  products of groups}, J. Algebra \textbf{171} (1995), no.~1, 230--257.
  \MR{1314099 (96a:20052)}

\bibitem{HR2011}
Derek~F Holt and Sarah Rees, \emph{{Generalising some results about
  right-angled Artin groups to graph products of groups}}, Preprint,
  \href{http://arxiv.org/abs/1110.2708v3}{arXiv:1110.2708 [math.GR]}.

\bibitem{HW1999}
Tim Hsu and Daniel~T. Wise, \emph{On linear and residual properties of graph
  products}, Michigan Math. J. \textbf{46} (1999), no.~2, 251--259. \MR{1704150
  (2000k:20056)}

\bibitem{KM2009}
Jeremy Kahn and Vladimir Markovic, \emph{{Immersing almost geodesic surfaces in
  a closed hyperbolic three manifold. {Preprint},
  \href{http://arxiv.org/abs/0910.5501}{arXiv:0910.5501v1}}}.

\bibitem{kim2007}
Sang-hyun Kim, \emph{Hyperbolic surface subgroups of right-angled {A}rtin
  groups and graph products of groups},  (2007), 144, Thesis (Ph.D.)--Yale
  University. \MR{2710318}

\bibitem{kim2008}
\bysame, \emph{Co-contractions of graphs and right-angled {A}rtin groups},
  Algebr. Geom. Topol. \textbf{8} (2008), no.~2, 849--868. \MR{2443098
  (2010h:20093)}

\bibitem{kim2010}
\bysame, \emph{On right-angled {A}rtin groups without surface subgroups},
  Groups Geom. Dyn. \textbf{4} (2010), no.~2, 275--307. \MR{2595093
  (2011d:20073)}

\bibitem{KO2011}
Sang-hyun Kim and Sang-il Oum, \emph{{Hyperbolic surface subgroups of one-ended
  doubles of free groups}}, Preprint,
  \href{http://arxiv.org/abs/1009.3820}{arXiv:1009.3820v3 [math.GR]}.

\bibitem{KW2009}
Sang-hyun Kim and Henry Wilton, \emph{Polygonal words in free groups}, to
  appear in Q. J. of Math., 2010.

\bibitem{LS2001}
Roger~C. Lyndon and Paul~E. Schupp, \emph{Combinatorial group theory}, Classics
  in Mathematics, Springer-Verlag, Berlin, 2001, Reprint of the 1977 edition.
  \MR{1812024 (2001i:20064)}

\bibitem{meier1996}
John Meier, \emph{When is the graph product of hyperbolic groups hyperbolic?},
  Geom. Dedicata \textbf{61} (1996), no.~1, 29--41. \MR{1389635 (97b:20055)}

\bibitem{SDS1989}
Herman Servatius, Carl Droms, and Brigitte Servatius, \emph{Surface subgroups
  of graph groups}, Proc. Amer. Math. Soc. \textbf{106} (1989), no.~3,
  573--578. \MR{952322 (90f:20052)}

\end{thebibliography}
\end{document}